\documentclass[11pt,reqno]{amsart}
\usepackage{xifthen}
\usepackage{subfig,csquotes }
\usepackage{pkgfile}

\newcommand{\ub}{\sM[k,\mu_0]}
\newcommand{\llb}{\mathcal{V}_n[k,\mu_0]}
\newcommand{\lbl}{\mathcal{V}[k,\mu_0]}

\title[Estimation in exponential families on permutations]{Estimation in  exponential families on permutations}

\author[S.\ Mukherjee]{Sumit Mukherjee}

\address{ Department of Statistics, Columbia University
\newline\indent 1255 Amsterdam Avenue, New York, NY 10027
\newline\indent Email id: {\tt sm3949@columbia.edu}}

\begin{document}

\date{\today}



\begin{abstract}
Asymptotics of the normalizing constant  is computed for a class of one parameter exponential families on permutations which includes  Mallows model with Spearmans's Footrule and Spearman's Rank Correlation Statistic. The MLE, and a computable approximation of the MLE are shown to be consistent. The pseudo-likelihood estimator of Besag is shown to be $\sqrt{n}$-consistent. An iterative algorithm (IPFP) is proved to converge to the limiting normalizing constant. The Mallows model with Kendall's Tau is also analyzed to demonstrate flexibility of the tools of this paper.
\end{abstract}

\maketitle

\section{Introduction}

Analysis of  permutation data has a fairly long history in statistics. One of the earlier papers in this area is the work of Mallows  (\cite{Mallows}) in 1957, where the author proposed a class of exponential families of permutations, thereby referred to in the literature as Mallows models. Using this modelling approach, in 1978 Feigin and Cohen( \cite{FCo})  analyzed the nature of  agreement between several judges in a contest. In 1985, Critchlow (\cite{Cr}) gave some examples where Mallows model gives a good fit to ranking data. See also the works of Fligner and Verducci (\cite{FV1,FV2}), and Critchlow, Fligner and Verducci  (\cite{CFV}), which deal with various aspects of permutation models, and the book length treatment of Marden in \cite{Mar}, which covers both theoretical and applied aspects of permutation modeling.
Permutation modeling has also received some recent attention in Machine Learning literature. Location and Scale mixture of Mallows model have been studied in \cite{ABSV,LL}. A generalized version of  Mallows model  which was introduced by Fligner and Verducci  was studied in \cite{CBBK,MPPB}, which has been extended to
infinite permutations in
\cite{MB1,MB2}.  The works of \cite{HGG,KHJ} study inference on permutations via fourier analysis of representation of finite groups with the focus of reducing computational complexity. Modeling of partially ranked data using Mallows models and its extensions was studied in \cite{LM}.

This paper analyzes a class of exponential families on the space of permutations  using the recently developed concept of  permutation limits. The notion of permutation limits has been first introduced in \cite{HKMRS}, and is motivated by dense graph limit theory (see \cite{BCLSV,BCLSV2,CD,LB} and the references there-in) . The main idea is that a permutation can be thought of as a probability measure on the unit square with uniform marginals. Multivariate distribution with uniform marginals have been studied widely in  Probability and Statistics (see \cite{GM,JDHR,MS,MN,RST,SW,Sklar,W} and references there-in) and  Finance (see \cite{ACH,BPT,DE,LR,MFE,Meu,MV,N} and references there-in) under the name copula.  One of the reasons for their popularity is that copulas are able to capture any dependence structure, as shown in Sklar's theorem (\cite{Sklar}). This is particularly useful in Finance when the assumption of independence of observations is far from the truth, as is believed for stock prices of various companies. 

%
%

\subsection{The 1970 Draft Lottery}

  To see how permutation data can arise naturally, consider the following example of historical importance where a random permutation was used to decide fate of human lives.   On December 1, 1969 during the Vietnam War  the U.S. Government used a random permutation of size $366$ to decide the relative dates of when the people  (among the citizens of U.S.A. born between the years 1944-1950) will be inducted into the army in the year 1970, based on their birthdays. 366 cylindrical capsules were put in a large box, one for each day of the year. The people who were born on the first chosen date  had to join the war first, those born on the second chosen date had to join next, and so on. 
 There were widespread allegations that the chosen permutation was not uniformly random. In \cite{F} Fienberg computed the Spearman's rank correlation between the birthdays and lottery numbers to be $-0.226$, which is significantly negative at $0.001$ level of significance. This suggests that people born in later part of the year were more likely to be  inducted earlier in the army.
 
  If a permutation is not chosen uniformly at random, then the question arises whether a particular non uniform model gives a better fit. It might be the case that there is a specific permutation $\sigma$ towards which the sampling mechanism has a bias, and permutations close to $\sigma$ have a higher probability of being selected.  For example in the draft lottery example $\sigma$ is the permutation 
  $$(366,365,364,\cdots,3,2,1).$$
  
  The Mallows models, which have the p.m.f. 
  $$e^{- \theta\sum\limits_{i=1}^nd(\pi,\sigma)-Z_n(\theta,\sigma)},$$
 are able to capture such behavior. Here $\sigma$ is a fixed permutation which is a location parameter, $\theta$ is a real valued parameter, and $d(.,.)$ is a distance function on the space of permutations. Here $Z_n(\theta,\sigma)$ denotes the (unknown)  log normalizing constant of this family. For $\theta$ large and positive, permutations $\pi$ which are away from $\sigma$   have small probability compared to those close to $\sigma$. The hypothesis of uniformity in this setting is equivalent to the hypothesis that $\theta=0$.

Possibly the most famous and widely used model on permutations is the Mallows model with Kendall's Tau as the metric. One of the reasons for this is that for this model the normalizing constant is known explicitly (see, for example \cite[(2.9)]{DR}), and so analyzing this model becomes a lot simpler. However when one moves away from Mallows model with Kendall's Tau and its generalizations, not much theory is available in the literature.  One reason for this is that normalizing constant is not available in closed form, and there is no straight forward independence assumptions in the model which one can exploit to analyze such models. Even basic properties for such models such as identifiability and consistency of estimates are not well understood.

 \subsection{Some common choices of Metric}
 \bigskip

By a metric $d(.,.)$ is usually meant a non negative function on $S_n\times S_n$ satisfying the following conditions:
  \begin{align*}
&d(\pi,\sigma)\ge 0,\text{ with equality iff }\pi=\sigma,\\
&d(\pi,\sigma)=d(\sigma,\pi),\\
&d(\pi,\sigma)\le d(\pi,\tau)+d(\sigma,\tau).
\end{align*}
Another restriction  on $d(.,.)$ which seems reasonable is that $d(.,.)$ is right invariant, i.e. 
\begin{align*}
d(\pi,\sigma)=d(\pi\circ \tau,\sigma\circ \tau),\text{ for all }\pi,\sigma, \tau\in S_n.
\end{align*} 
The justification for this last requirement is as follows: Suppose the students in a class are labelled $\{1,2,\cdots,n\}$, and let $\pi(i)$ and $\sigma(i)$ denote the rank of student $i$  based on Math and Physics scores respectively (assume no tied scores). The distance $d(\pi,\sigma)$ can be thought of as a measure of the strength of the relationship between Math and Physics rankings. If students are now labelled differently using  a permutation $\tau$, so that student $i$ now becomes student $\tau(i)$, then the Math and Physics rankings become $\pi\circ\tau$ and $\sigma\circ\tau$ respectively. But this relabeling of students in principle should not change the relation between Math and Physics rankings, which requires the right invariance of  $d(.,.)$. 

 Some of the common choices  of right invariant metric $d(.,.)$   are the following (\cite[Ch-5,6]{D}). 

\begin{enumerate}[(a)]
\item{Spearman's Foot Rule:}  
$\sum_{i=1}^n|\pi(i)-\sigma(i)|$
\\

\item{ Spearman's Rank correlation:} $\sum_{i=1}^n(\pi(i)-\sigma(i))^2$
\\

\item{Hamming Distance:} 
$\sum_{i=1}^n1\{\pi(i)\neq \sigma(i)\}$ 
\\

\item{Kendall's Tau:} Minimum number of pairwise adjacent transpositions which converts $\pi^{-1}$ into $\sigma^{-1}$.
\\

\item{Cayley's distance:} Minimum number of  adjacent transpositions which converts $\pi$ into $\sigma$
=$n-$ number of cycles in $\pi\sigma^{-1}$.
\\

\item{Ulam's distance:} 
Number of deletion-insertion operations to convert $\pi$ into $\sigma=n-$ Length of the longest increasing subsequence in $\sigma\pi^{-1}$.
\end{enumerate}
See \cite[Ch-5,6]{D} for more details  on these metrics.  It should be noted that the Spearman's Rank correlation term is the square of a metric and does not satisfy the triangle inequality, but this version is used  in the literature as it is right invariant.
If $d(.,.)$ is right invariant, then the normalizing constant is free of $\sigma$, as
$$\sum_{\pi\in S_n}e^{-\theta d(\pi,\sigma)}=\sum_{\pi\in S_n}e^{-\theta d(\pi \circ\sigma^{-1},e)}=\sum_{\pi \in S_n}e^{-\theta d(\pi,e)},$$
where $e$ is the identity permutation. Also if $\pi$ is a sample from the probability mass function $e^{-\theta d(\pi,\sigma)-Z_n(\theta)}$, then $\pi\circ \sigma^{-1}$ is a sample from the probability mass function $e^{-\theta d(\pi,e)-Z_n(\theta)}$. This paper focuses on the case where $\sigma$ is known, and carries out inference on $\theta$ when one sample $\pi$  is observed from this model. If the location parameter $\sigma$ is unknown, estimating it from one permutation $\pi$ seems impossible, unless the model puts very small mass on permutations which are away from $\sigma$, in which case $\pi$ itself is a reasonable estimate for $\sigma$. In case $\sigma$ is known, without loss of generality by a  relabeling it can be assumed that $\sigma$ is the identity permutation. 
In an attempt to cover the first two metrics  in the above list, consider an exponential family of the form
\begin{align}\label{thesis2}
\Q_{n,f,\theta}(\pi)=e^{\theta \sum_{i=1}^nf(i/n,\pi(i)/n)-Z_n(f,\theta)},
\end{align}
where $f$ is a continuous function on the unit square. In particular, if $f(x,y)=-|x-y|$ then 
$$\sum_{i=1}^nf(i/n,\pi(i)/n)=-\frac{1}{n}\sum_{i=1}^n|i-\pi(i)|,$$
which is a scaled version of the Foot rule  (see (a) in list above). For the choice $f(x,y)=-(x-y)^2$,
$$\sum_{i=1}^nf(i/n,\pi(i)/n)=-\frac{1}{n^2}\sum_{i=1}^n(i-\pi(i))^2$$
is a scaled version of Spearman's rank correlation statistic (see (b) in the list above). A simple calculation shows that the right hand side above is same as $$\frac{(n+1)(2n+1)}{3n}+\frac{2}{n^2}\sum_{i=1}^ni\pi(i),$$
and so the same model would have been obtained by setting $f(x,y)=xy$. Note that the Hamming distance (third in the list of metrics) is also of this form for the choice $f(x,y)=1_{x\neq y}$ which is a discontinuous function.

\begin{remark}
It should be noted here that the model $\Q_{n,f,\theta}$ covers a wide class of models, some of which are not unimodal. For e.g. if one sets $f(x,y)=x(1-x)y$ then for $n=7$
\begin{align*}
\sum_{i=1}^7f(i/7,j/7)=\frac{1}{7^3}\sum_{i=1}^7i(7-i)\pi(i),
\end{align*}
which is maximized when
$$\pi(7)=1,\quad \{\pi(1),\pi(6) \}=\{2,3\},\quad \{\pi(2),\pi(5) \}=\{4,5\},\quad \{\pi(3),\pi(4) \}=\{6,7\}$$
Thus for $\theta>0$ this model has $2^3=8$ modes. In general for $\theta>0$ this model has $2^{(n-1)/2}$ modes for $n$ odd, and $2^{(n-2)/2}$ modes for $n$ even.

If one assumes that for every fixed $y$ the function $x\mapsto f(x,y)$ has a unique global maximum at $x=y$, then the model $\Q_{n,f,\theta}$ is unimodal. Indeed, in this case the mode is the identity permutation $(1,2,\cdots,n)$ for $\theta>0$ and the reverse identity permutation $(n,n-1,\cdots,1)$ for $\theta<0$. Note that both the functions $f(x,y)=-(x-y)^2$ and $f(x,y)=-|x-y|$ satisfy this condition.
\end{remark}

One important comment about the model $\Q_{n,f,\theta}$ is that different choices of the function $f$ may give to the same model. Indeed as already remarked above, the function $f(x,y)=-(x-y)^2/2$ and $f(x,y)=xy$ gives rise to the same model. In general whenever $f(x,y)-g(x,y)$ can be written as $\phi(x)+\psi(y)$ for any two functions $\phi,\psi:[0,1]\mapsto \R$ the two models are the same.
In particular, the function $f(x,y)=x+y$ and $g(x,y)\equiv 0$ gives rise to the same model, which is the uniform distribution on $S_n$. The following definition
restricts the class of functions $f$ to ensure identifiability.
\begin{defn}
Let $\cC$ be the set of all continuous functions $f$ on $[0,1]^2$ which satisfy
\begin{align}\label{cnc}
\int_0^1 f(x,z)dz=0,\forall x\in [0,1];\quad \int_0^1 f(z,y)dz=0,\forall y\in [0,1],
\end{align}
 and $f$ is not identically $0$.
\end{defn}
Another set of constraints which would have served the same purpose is $f(x,0)=0,\forall x\in [0,1]; f(0,y)=0,\forall y\in [0,1]$. For the sake of definiteness this paper uses \eqref{cnc}. This mimics the condition in the discrete setting that the row and column sums of a square matrix are all $0$. It should be noted here that the function $f(x,y)=xy$ does not belong to $\cC$, and it should be replaced by the function $f(x,y)=(x-1/2)(y-1/2)$. However this is not done in sections 2 and 3 to simplify notations, on observing  that all the proofs and conclusions of this paper go through as long as $f(x,y)$ cannot be written as $\phi(x)+\psi(y)$, which is true for $f(x,y)=xy$.
\subsection{Statement of main results}
The first main result of this paper is the following theorem which computes the limiting value of the log normalizing constant of models of the form \eqref{thesis2}  for a general continuous function $f$ in terms of an optimization problem over copulas.

\begin{defn}
	Let $\sM$ denote the space of all probability distributions on the unit square with uniform marginals.
	\end{defn}
	

\begin{thm}\label{l1}
	For any function $f\in\mathcal{C}$  consider the probability model $\mathbb{Q}_{n,f,\theta}(\pi)$ as defined in \eqref{thesis2}, and $\theta\in \R$ is fixed. Then the following conclusions hold:
	
	\begin{enumerate}
		\item[(a)]
		
		\[\lim_{n\rightarrow\infty}\frac{Z_n(f,\theta)-Z_n(0)}{n}=Z(f,\theta):=\sup_{\mu\in \sM}\{\theta\mu[f]-D(\mu||u)\},\]
		where $u$ is the uniform distribution on the unit square,  $\mu[f]:=\int fd\mu$ is the expectation of $f$ with respect to the measure $\mu$, and $D(.||.)$ is the Kullback-Leibler divergence.

		\item[(b)]
		If $\pi\in S_n$ is a random permutation from the model $\Q_{n,f,\theta}$, then the  random probability measure  $$\nu_\pi:=\frac{1}{n}\sum_{i=1}^n\delta_{\Big(\frac{i}{n},\frac{\pi(i)}{n}\Big)}$$
		on the unit square converge  weakly in probability to the probability measure $\mu_{f,\theta}\in \sM$, where 
		$\mu_{f,\theta}$  the unique maximizer of part (a). 
		
	\item[(c)]	
		The measure $\mu_{f,\theta}$ of part (b) has density  $$g_{f,\theta}(x,y):=e^{\theta f(x,y)+a_{f,\theta}(x)+b_{f,\theta}(y)}$$  with respect to Lebesgue measure on $[0,1]^2$, with the functions $a_{f,\theta}(.),b_{f,\theta}(.)\in L^1[0,1]$ which are unique almost surely.  Consequently one has \[\sup_{\mu\in\sM}\{\theta \mu[f]-D(\mu||u)\}=-\int_{x=0}^1[a_{f,\theta}(x)+b_{f,\theta}(x)]dx.\]
		
		\item[(d)]
		The function $Z(f,\theta)$ of part (b) is a  differentiable convex function with  a continuous and strictly increasing derivative $Z'(f,\theta)$ which satisfies $$Z'(f,\theta)=\lim_{n\rightarrow\infty}\frac{1}{n}Z_n'(f,\theta)=\mu_{f,\theta}[f].$$ 
	
	\end{enumerate}
\end{thm}
\begin{remark}
Part (b) of the above theorem gives one way to visualize a permutation $\pi$ as a measure $\nu_\pi$ on the unit square. The appendix gives a somewhat similar way to view a permutation $\pi$ as a measure $\mu_\pi$. It also demonstrates how the measure $\nu_\pi$ looks like, when $\pi$ is a large permutation from $\Q_{n,f,\theta}$. As an example, setting $\theta=0$ one gets the uniform distribution on $S_n$, when the limiting measure $\mu_{f,\theta}$ becomes $u$ the uniform distribution on $[0,1]^2$.
Note that the theorem statement uses $Z_n(0)$ instead of $Z_n(f,0)$. This is because
$Z_n(f,0)=\log n!$ for all choices of the function $f$, and so the use of the notation $Z_n(0)$ is without loss of generality. 
\end{remark}

Focusing on inference about $\theta$ when an observation $\pi$ is obtained from the model $\Q_{n,f,\theta}$, then the following corollary of theorem \ref{l1} shows consistency of the Maximum Likelihood Estimate. In this model MLE for $\theta$ is the solution to the equation $$\Big\{\frac{1}{n} \sum\limits_{i=1}^nf(i/n,\pi(i)/n)-\frac{1}{n}Z_n'(f,\theta)\Big\}=0.$$ Since $Z_n(f,\theta)$ and $Z_n'(f,\theta)$ is hard to compute numerically, as an approximation one can replace the quantity
$\frac{1}{n}Z_n'(f,\theta)$ above by its limiting value $Z'(f,\theta)$ and then solve for $\theta$.  The following corollary shows that this estimate is consistent for $\theta$ as well.

\begin{cor}\label{thm:ldmle}
For $f\in \cC$ consider the model $\Q_{n,f,\theta}$ as in \eqref{thesis2}, and let  $\pi$ be an observations from this model. 
\begin{enumerate}[(a)]
\item
In this case  one has
$$\frac{1}{n}\sum_{i=1}^nf(i/n,\pi(i)/n)\stackrel{P}{\rightarrow}Z'(f,\theta)=\mu_{f,\theta}[f]$$ for every $\theta\in\R$.
		
		\item
		Both the expressions
		\begin{align*}ML_n(\pi,\theta):=&\frac{1}{n} \sum\limits_{i=1}^nf(i/n,\pi(i)/n)-Z_n'(f,\theta)\\
		LD_n(\pi,\theta):=&\frac{1}{n} \sum\limits_{i=1}^nf(i/n,\pi(i)/n)-Z'(f,\theta)
		\end{align*} have unique roots $\hat{\theta}_{ML}$ and $\hat{\theta}_{LD}$ with probability tending to $1$ which are consistent for $\theta$.

\item
Consider the testing problem of $\theta=\theta_0$ versus $\theta=\theta_1$ with $\theta_1> \theta_0$. Then the test $\phi_n:=1\{\hat{\theta}_{LD}>(\theta_0+\theta_1)/2\}$ is consistent, i.e.
$$\lim_{n\rightarrow\infty}\E_{\Q_{n,f,\theta_0}}\phi_n=0,\quad \lim_{n\rightarrow\infty}\E_{\Q_{n,f,\theta_1}}\phi_n=1.$$
\end{enumerate}
\end{cor}
The above corollary  shows that it is possible to estimate the parameter $\theta$ consistently with just one observation from the model $\Q_{n,f,\theta}$. No error rates can be obtained for the estimates $\{\hat{\theta}_{ML},\hat{\theta}_{LD}\}$ as part (a) of theorem \ref{l1} does not have any error rates.
 Thus a good approximation of the limiting log normalizing constant will lead to an efficient estimator for $\theta$, in the sense that the estimator will be close to the MLE. The definition of $Z(f,\theta)$
is in terms of an optimization problem over $\sM$, which is an infinite dimensional space. In general, such optimization can be hard to carry out. The next theorem  gives an iterative algorithm for computing the density of the optimizing measure $\mu_{f,\theta}$ with respect to Lebesgue measure.
Intuitively the algorithm starts with the function $e^{\theta f(x,y)}$ and alternately scales it along $x$ and $y$ marginals to produce uniform marginals in the limit. 
\begin{defn}\label{def:mk}
For any integer $k\ge 1$ let $\cM_k$ denote the set of all $k\times k$ matrices with non negative entries with both row and column sums equal to $1/k$. 
\end{defn}

\begin{thm}\label{approximate}
	\begin{enumerate}[(a)]
	\item
		Define a sequence  of $k\times k$ matrices by setting $B_0(r,s):=e^{f(r/k,s/k)}$ for $1\le r,s\le k$, and 
		 \[B_{2m+1}(r,s):=\frac{B_{2m}(r,s)}{k\sum_{l=1}^mB_{2m}(r,l)},B_{2m+2}(r,s):=\frac{B_{2m+1}(r,s)}{k\sum_{l=1}^mB_{2m+1}(l,s)}.\] 
		Then there exists a matrix $A_{k,\theta}\in \cM_k$ such that
		$\lim_{m\rightarrow\infty}B_m=A_k$.
		
		\item
		 $A_{k,\theta}\in \cM_k$ is the unique maximizer of the optimization problem
		$$\sup_{A\in \cM_k}\Big\{\theta \sum_{r,s=1}^kf(r/k,s/k)A(r,s)-2\log k-\sum_{r,s=1}^kA(r,s)\log A(r,s)\Big\}.$$
		
		\item
		The function $$W_k(f,\theta):=\sup_{A\in \cM_k}\Big\{\theta\sum_{r,s=1}^kf(r/k,s/k)A(r,s)-2\log k-\sum_{r,s=1}^kA(r,s)\log A(r,s)\Big\}.$$
		
		is  a convex differentiable function with $$W_k'(f,\theta)= \sum_{r,s=1}^kA_{k,\theta}(r,s)f(r/k,s/k).$$

		\item[(d)]
		
		Finally, for any continuous function $\phi:[0,1]^2\mapsto \R$ one has 
		$$\lim_{k\rightarrow\infty}\sum_{r,s=1}^kA_{k,\theta}(r,s)\phi(r/k,s/k)=\int_{[0,1]^2}\phi(x,y)g_{f,\theta}(x,y)dxdy.$$ In particular this implies
	$$Z(f,\theta)=\lim_{k\rightarrow\infty}W_k(f,\theta)=\lim_{k\rightarrow\infty}\lim_{m\rightarrow\infty}\Big\{\theta\sum_{i,j=1}^kf(i/k,j/k)B_m(i,j)-2\log k-\sum_{i,j=1}^k B_m(i,j)\log B_m(i,j)\Big\} .$$
	\end{enumerate}
\end{thm}

\begin{remark}\label{april}
	Since $g_{f,\theta}(x,y)=e^{a_{f,\theta}(x)+b_{f,\theta}(y)+\theta f(x,y)}$ has uniform marginals, the functions $a_{f,\theta}(.)$ and $b_{f,\theta}(.)$ are the solutions to the joint integral equations
	\[\int_{0}^1e^{\theta f(x,z)+a_{f,\theta}(x)+b_{f,\theta}(z)}dz=1,\quad \int_{0}^1e^{\theta f(z,y)+a_{f,\theta}(z)+b_{f,\theta}(y)}dz=1, \text{ for all }x,y\in [0,1].\] By  theorem  \ref{approximate}, it follows that \[\lim_{n\rightarrow\infty}\frac{Z_n(f,\theta)-Z_n(0)}{n}=-\int_{x=0}^1[a_{f,\theta}(x)+b_{f,\theta}(x)]dx.\]
	For the limiting normalizing constant in the Mallows model with the Foot-rule or the Spearman's rank correlation one needs to take $f(x,y)=-|x-y|$ and $f(x,y)=-(x-y)^2 (\text{ or } f(x,y)=xy)$ respectively. Even though analytic computation for $a_{f,\theta}(.),b_{f,\theta}(.)$ might be difficult, the algorithm of theorem \ref{approximate} (known as IPFP) can be used for a numerical evaluation of these functions. Iterative Proportional Fitting Procedure (IPFP) originated  in the works of  Deming and Stephan (\cite{DS}) in 1940. For more background on IPFP see \cite{Cs,Kb,Ru,Sink} and the references there-in. Theorem \ref{approximate} gives a way to approximate numerically the limiting log partition function by fixing $k$ large and running the IPFP for $m$ iterations with a suitably large $m$.
\end{remark}

Another approach for estimation in such models can be to estimate the parameter $\theta$ without estimating the normalizing constant. The following theorem constructs an explicit $\sqrt{n}$ consistent estimator for $\theta$, for the class of models considered in theorem \ref{l1}.
This estimate is similar in spirit to Besag's pseudo-likelihood estimator \cite{B1,B2}.  The pseudo-likelihood is defined to be the product of all  one dimensional conditional distributions, one for each random variable. 
Since in a permutation the conditional distribution of $\pi(i)$ given $\{\pi(j),j\ne i\}$  determines the value of  $\pi(i)$, it does not make sense to look at the conditional distribution  $(\pi(i)|\pi(j),j\neq i)$. In this case a meaningful thing  to consider is the distribution of $(\pi(i),\pi(j)|\pi(k),k\neq i,j)$, which gives the pseudo-likelihood as 
\[\prod_{1\le i<j\le n}\Q_{n,f,\theta}(\pi(i),\pi(j)|\pi(k),k\ne i,j).\]
The pseudo-likelihood estimate  $\hat{\theta}_{PL}$ is obtained by maximizing the above expression. Taking the $\log$ of the pseudo-likelihood and differentiating with respect to $\theta$ gives 
\[\sum_{1\le i<j\le n}y_\pi(i,j)\frac{1}{1+e^{\theta y_\pi(i,j)}},\]
where $$y_\pi(i,j):=f(i/n,\pi(i)/n)+f(j/n,\pi(j)/n)-f(i/n,\pi(j)/n)-f(j/n,\pi(i)/n).$$
The pseudo-likelihood estimate can then be obtained by equating this to $0$ and solving for $\theta$. One way of computing this estimate is a grid search over $\R$ and does not require the computation of $Z_n(f,\theta)$. Thus  this gives a fast and practical way for parameter estimation in such models. The next theorem gives error rates for the pseudo-likelihood estimator.
\begin{thm}\label{est}
	
		For  $f\in \cC$ consider the model $\Q_{n,f,\theta}$ of \eqref{thesis2}, and let $\pi$ be a sample from $\Q_{n,f,\theta}$.
		Setting $y_\pi(i,j)=:f(i/n,\pi(i)/n)+f(j/n,\pi(j)/n)-f(i/n,\pi(j)/n)-f(j/n,\pi(i)/n)$ the expression 
				\begin{align*}
		PL_n(\pi,\theta):=\sum_{1\le i<j\le n}y_\pi(i,j)
		 \frac{1}{1+e^{\theta y_\pi(i,j)}}.
		\end{align*}
		 has a unique root  in $\theta$ with probability tending to $1$. 
		Further, denoting this root by $\hat{\theta}_n$ one has    $\sqrt{n}(\hat{\theta}-\theta)$ is $O_P(1)$.
	
\end{thm}
The estimating equations $LD_n(\pi,\theta),ML_n(\pi,\theta)$ of Corollary \ref{thm:ldmle} and $PL_n(\pi,\theta)$ of theorem \ref{est} are stated when a single permutation $\pi$ is observed from $\Q_{n,f,\theta}$. If i.i.d. samples $\pi^{(1)},\cdots,\pi^{(m)}$ are observed from $\Q_{n,f,\theta}$ one should use the equations $$\sum_{l=1}^mLD_n(\pi^{(l)},\theta),\quad  \sum_{l=1}^m ML_n(\pi^{(l)},\theta), \quad\sum_{l=1}^mPL_n(\pi^{(l)},\theta)$$ instead.

So far all results  relate to the model $\Q_{n,f,\theta}$ as defined in (\ref{thesis2}). To demonstrate that the tools used to prove these results are quite robust, the next proposition analyzes the Mallows model with Kendall's Tau as the metric (item (d) in the original list of metrics). 
\begin{ppn}\label{mallow}
	Consider the Mallows model on $S_n$ with Kendall's tau as the metric, defined by  \[M_{n,\theta}(\pi)=e^{\frac{\theta}{n} Inv(\pi)-C_n(\theta)},\quad Inv(\pi):=\sum_{i<j}1_{\pi(i)>\pi(j)},\] 
	where $C_n(\theta)$ is the normalizing constant. Also let $h:[0,1]^4\mapsto \R$ denote the function
	$$ h((x_1,y_1),(x_2,y_2)):=1_{(x_1-x_2)(y_1-y_2)<0},$$
	
	\begin{enumerate}[(a)]
	\item
	In this case one has
	\[\lim_{n\rightarrow\infty}\frac{C_n(\theta)-C_n(0)}{n}=C(\theta):=\sup_{\mu\in \sM}\Big\{\frac{\theta}{2}(\mu\times\mu)(h)-D(\mu||u)\Big\}.\]

Further, 
	the supremum above is attained at a unique measure on the unit square given by the density
	$$\rho_\theta(x,y):=\frac{\frac{\theta}{2}\sinh\frac{\theta}{2}}{\Big[e^{-\frac{\theta}{4}}\cosh(\frac{\theta(x-y)}{2})+e^{\frac{\theta}{4}}\cosh(\frac{\theta(x+y-1)}{2})\Big]^2},$$
	and consequently $C(\theta)=\int_0^1 \frac{e^{\theta x}-1}{\theta x}dx$.

	\item
	If $\pi$ is a sample from $M_{n,\theta}$, then both the  expressions
	\begin{align*}
	\widetilde{ML}_n(\pi,\theta):=& \frac{1}{n^2}Inv(\pi)-C_n'(\theta),\\
	\widetilde{LD}_n(\pi,\theta):=& \frac{1}{n^2}Inv(\pi)-C'(\theta)
	\end{align*}
	have unique roots $\widetilde{\theta}_{ML}$ and $\widetilde{\theta}_{LD}$ with probability tending to $1$ which are consistent for $\theta$.
	
	\end{enumerate}

	
	

\end{ppn}
\begin{remark}

Since for the Mallows model with Kendall's Tau as metric the partition function $C_n(\theta)$ is explicitly known, the formula for $C(\theta)$ can be computed easily. In this case by a direct argument one can show that $\widetilde{\theta}_{LD},\widetilde{\theta}_{ML}$ are $\sqrt{n}$ consistent. The theorem shows that the general tools developed in this paper can also be used to show consistency, even though establishing optimal rates requires finer results.
\end{remark}

Even though the Mallows model with Kendall's Tau is not in the setting of Theorem \ref{l1}, estimation of the log normalization constant is still possible using results of this paper. This is 
because the
function $$\mu\mapsto \int_{[0,1]^4}1_{(x_1-x_2)(y_1-y_2)<0}d\mu(x_1,y_1)d\mu(x_2,y_2)$$ is continuous on $\sM$ with respect to weak convergence, and 
is a natural extension for the number of inversions of a permutation to a general probability measure  in $\sM$.
Thus to explore other non uniform models on permutations, one needs to understand the continuous real valued functionals on $\sM$. 
For an example of a natural function on permutations which is not continuous, let $N(\pi)$ denote the number of fixed points of $\pi$.  Then the function $\pi\mapsto N(\pi)/n$ is not continuous  on $\sM$.  
Indeed, its natural analogue on $\sM$ is the function
$$\mu\mapsto \int_{[0,1]^2}1_{x=y}d\mu(x,y),$$
which is not continuous with respect to weak topology on $\sM$. 
\\

Another interesting problem is to compute the limiting distribution of $\sum_{i=1}^nf(i/n,\pi(i)/n)$ under the model $\Q_{n,f,\theta}$. Under uniform distribution on $S_n$ this statistic has a limiting normal distribution if $f\in \cC$, by Hoeffding's combinatorial Central Limit Theorem (\cite[Theorem 3]{Hoeffding}). Theorem \ref{l1} shows that $\frac{1}{n}\sum_{i=1}^nf(i/n,\pi(i)/n)$ converges to a constant, and gives a characterization of this constant in terms of permutation limits. It however fails to find non-degenerate limit distribution for this statistic. If one is interested in the testing problem of $\theta=\theta_0$ versus $\theta=\theta_1$ as in corollary  \ref{thm:ldmle}, then this distribution will be useful in determination of exact cut offs under null hypothesis, and evaluation of power under the alternative. Also using such distribution results, it should be possible to find out limit distributions of the estimators considered in this paper.
\\

Finally, this paper explores the asymptotics of parametric models on permutations. Viewing a permutation as a measure one can study non parametric models on permutons as well, and in fact one class of models was introduced and studied in \cite{HKMRS}.  
Such models can be used to fit permutations. This technique can also be used for comparing permutations  in a non parametric manner, such as in a classification problems on permutations. Section 3 gives a visual comparison, but comparisons can also be carried out in a more precise manner using a \enquote{suitable} metric for bivariate probability measures.


\subsection{Main contributions}

This paper gives a framework for analyzing probability distribution on large permutations. It computes asymptotics of normalizing constants in a class of exponential families on permutations, and explores identifiability of such models. It derives the limit in probability of statistics  under such models, and shows the consistency of MLE and an estimate based on the limiting log normalizing constant for such models.
It gives an Iterative Proportional Fitting Procedure (IPFP) to numerically compute the normalizing constant. It also shows $\sqrt{n}$ consistency of the pseudo-likelihood estimator of Besag.  It demonstrates the flexibility of this approach by analyzing the Mallows model with Kendall's Tau as its metric.   For the Mallows model with Kendall's Tau, it again shows consistency of the MLE, and an estimate based on the limiting log normalizing constant.

The main tool for proving the results is a large deviation principle for a uniformly random permutation. An arxiv version of this paper uses the recently developed notion of permutation limits from \cite{HKMRS} to give a new proof of this large deviation principle.
\subsection{Outline}

Section \ref{sec:eg} explores the Mallows model with Spearman's rank correlation as sufficient statistic, using the results of this paper. Section \ref{sec:draft}  analyzes the draft lottery data of 1971. Appendix \ref{appen:1} 
describes in brief the concept of permutation limits introduced in \cite{HKMRS}, and proves a large deviation principle for permutations in theorem \ref{ldp}.
  It also   carries out the proofs of the main results of this paper  using theorem \ref{ldp}. 

\section{An example: Spearman's rank correlation metric}\label{sec:eg}
 This section illustrates the conclusions of Theorem \ref{l1} and Theorem \ref{approximate} with a concrete example, the Spearman's rank correlation model. This is number (b) in the list of metrics in the introduction, the Spearman's rank correlation  given by
 $$||\pi-\sigma||_2^2=\sum_{i=1}^n(\pi(i)-\sigma(i))^2.$$  As pointed out in the introduction this does not satisfy triangle inequality and so is not a metric in the proper sense. However this version is used as it is right invariant and algebraically more tractable, and as such has received attention in Statistics literature (see \cite{Cr,D,FCo,Mallows} and references therein). The reason for its nomenclature is that if $\pi$ and $\sigma $ are two permutations of size $n$, then the simple correlation coefficient of the points $\{(\pi(i),\sigma(i)\}_{i=1}^n$ has the formula
 $$r(\pi,\sigma)=1-\frac{6||\pi-\sigma||_2^2}{n(n^2-1)},$$
 which is a one-one function of $||\pi-\sigma||_2^2$.
 \\

 Even for this simple metric the normalizing constant for the corresponding Mallows model  is not available in closed form. 
 As observed in the introduction, the Spearman's rank correlation model is obtained by setting $f(x,y)=-(x-y)^2$ or $f(x,y)=xy$ in the model of Theorem \ref{l1}.
  This section will work with the choice $f(x,y)=xy$. To be precise, the p.m.f. of this model is
  $$\Q_{n,f,\theta}=e^{(\theta/n^2)\sum_{i=1}^ni\pi(i)-Z_n(f,\theta)},$$
  where $Z_n(f,\theta)$ is the appropriate log normalizing constant as before.  
By the discussion after equation \eqref{cnc} it follows that theorem \ref{l1} is applicable for $f(x,y)=xy$. Thus  if $\pi$ is a random permutation from $\Q_{n,f,\theta}$ for this $f$, then the empirical measure $\nu_\pi=\frac{1}{n}\sum_{i=1}^n\delta_{(i/n,\pi(i)/n)}$ converges weakly to a measure in $\cM$ with density of the form $$g_{f,\theta}(x,y)=e^{\theta xy+a_{f,\theta}(x)+a_{f,\theta}(y)},$$ where the symmetry of $f$ has been used to deduce $b_{f,\theta}(.)=a_{f,\theta}(.)$. 
Using the uniform marginal condition gives 
$$1=\int_{y=0}^1e^{\theta xy+a_{f,\theta}(x)+a_{f,\theta}(y)}dy=e^{a_{f,\theta}(x)}\sum_{k=0}^\infty C_k(\theta)\frac{x ^k\theta^k}{k!}$$ with $C_k(\theta):=\int_{0}^1 y^kh_{f,\theta}(y)dy$, and so 
$$e^{a_{f,\theta}(x)}=\Big(\sum_{k=0}^\infty C_k(\theta) \frac{x^k\theta^k}{k!}\Big)^{-1}.$$ 
Another integration with respect to $x$ gives
$$\sum_{k=0}^\infty \frac{C_k^2(\theta)\theta^k}{k!}=1.$$
However, analytic solution of $g_{f,\theta}(.)$ seems intractable and is not attempted here.
Instead,  figure \ref{limit_vs_sample}(a) plots the  density $g_{f,\theta}(x,y)=e^{\theta xy+a_{f,\theta} (x)+a_{f,\theta} (y)}$  on a discrete grid of size $k\times k$ with $k=1000$. The values of the function are computed by the algorithm of theorem \ref{approximate}   starting with the $k\times k$ matrix  $B_0(i,j)=e^{(\theta/k^2) ij}$, where $\theta=20$. Part (d) of \ref{approximate} implies that $k^2B_m(i,j)$ can be taken as an approximation of the limiting density $g_{f,\theta}(i/k,j/k)$.
\begin{figure*}[h]
\centering
\begin{minipage}[c]{0.5\textwidth}
\centering
\includegraphics[width=3.3in]
    {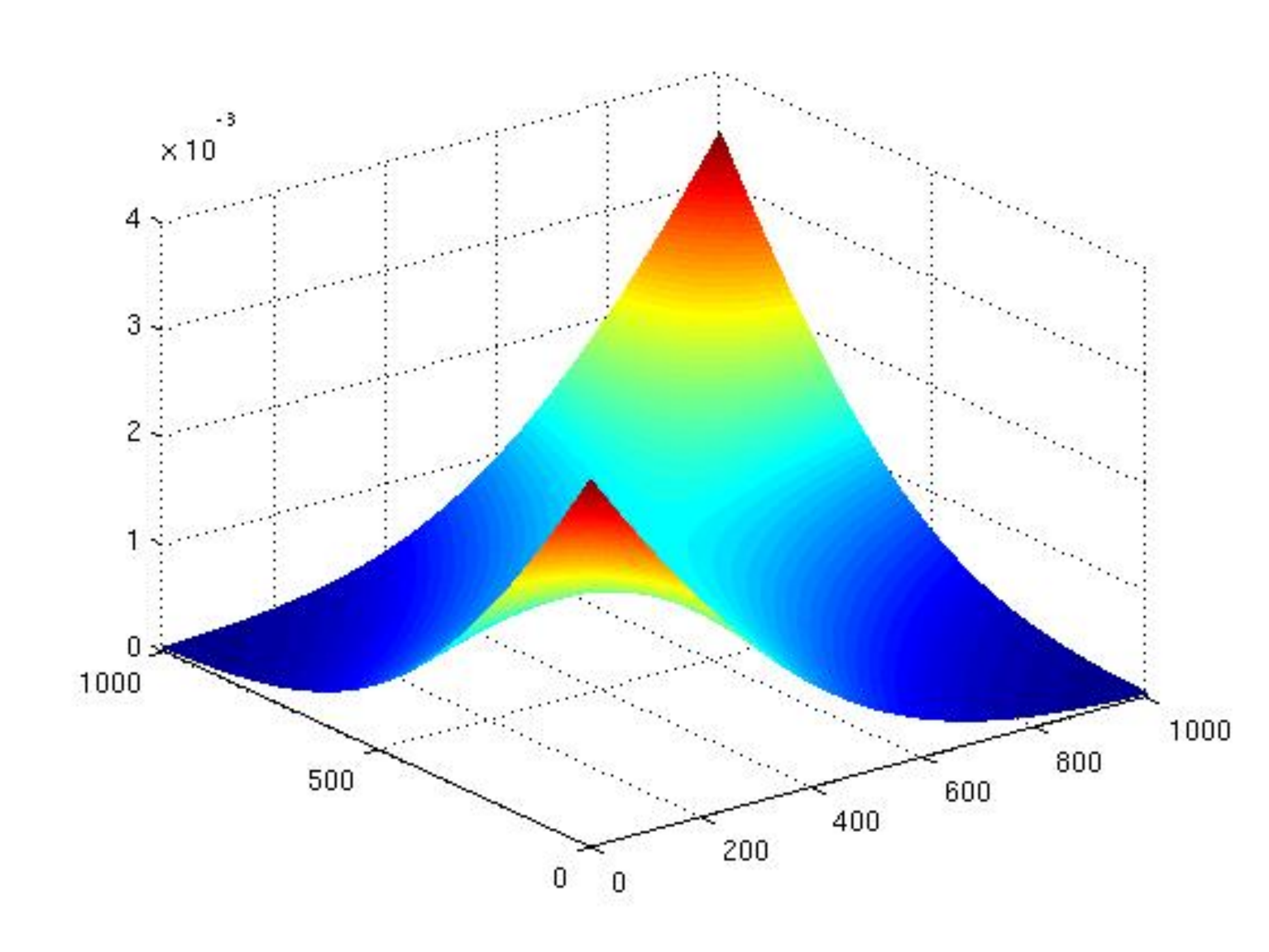}\\
\small{(a)}    
\end{minipage}
\begin{minipage}[c]{0.49\textwidth}
\centering
\includegraphics[width=3.5in]
    {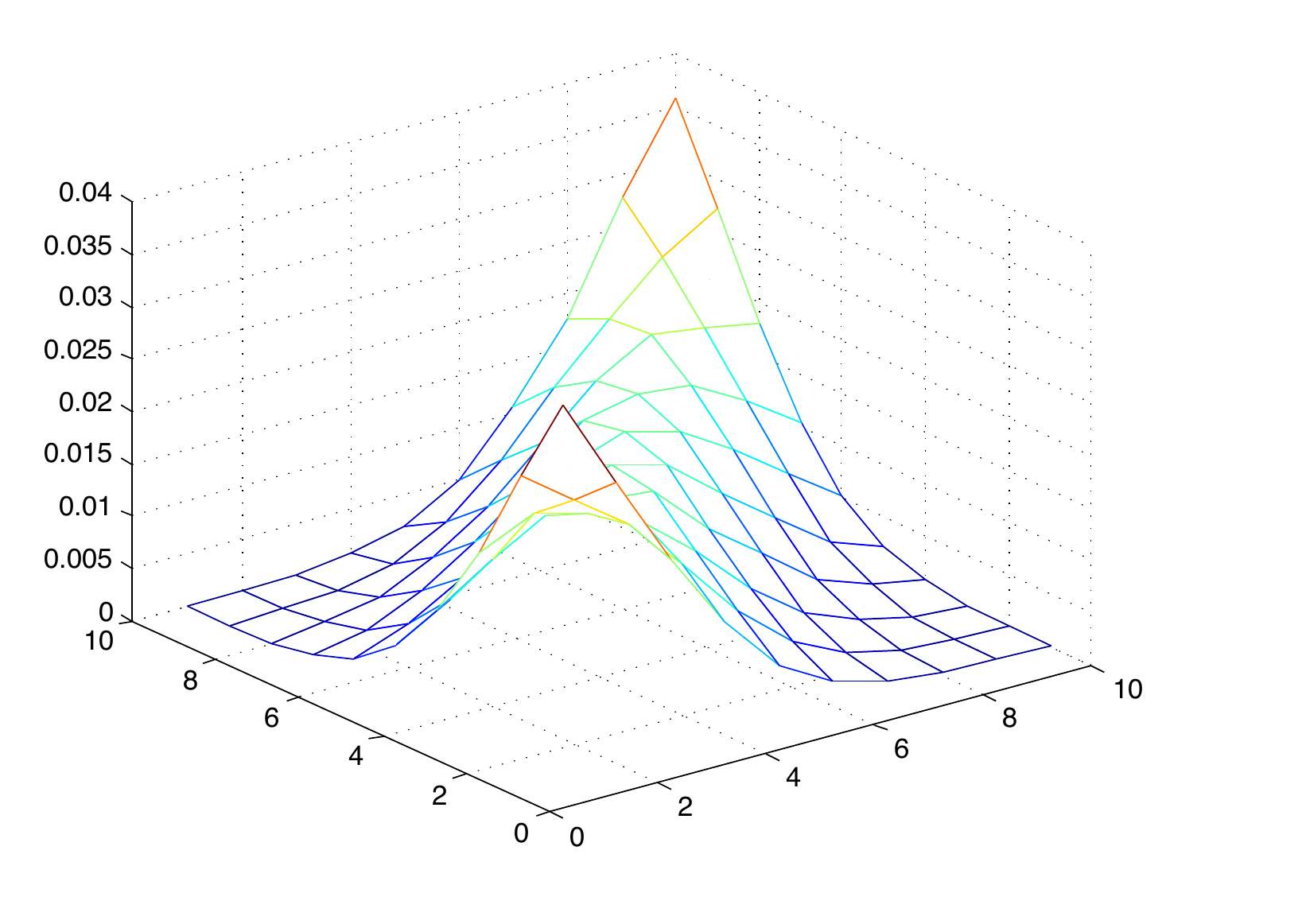}\\
\small{(b)}   
\end{minipage}
\caption{\small{(a) Density of limiting measure $\mu_{f,\theta}$ for $f(x,y)=xy,\theta=20$, (b)Histogram  of $\nu_\pi$ with $n=10000$ and $10\times 10$ bins.}}
\label{limit_vs_sample}
\end{figure*}

From figure \ref{limit_vs_sample} it is easy to see that $g_{f,\theta}$ has higher values on the diagonal $x=y$, which also follows from the fact that for $\theta>0$ the identity permutation has the largest probability under this model. The function $g_{f,\theta}(.,)$ is symmetric about the diagonal $x=y$, which follows from the fact that  $f(.,.)$ is symmetric. Another way to see this is by noting that if $\pi$ converges to a probability measure on $[0,1]^2$ with limiting density $g_{f,\theta}(x,y)$, then $\pi^{-1}$ converges to a measure on $[0,1]^2$ with limiting density $g_{f,\theta}(y,x)$. But since
$$\sum_{i=1}^ni\pi(i)=\sum_{i=1}^ni\pi^{-1}(i),$$
the law of $\pi$ and $\pi^{-1}$ are same under $\Q_{n,f,\theta}$, and so $\pi^{-1}$ has the limiting density $g_\theta(x,y)$ as well, thus giving $g_{f,\theta}(x,y)=g_{f,\theta}(y,x)$.

The function is also symmetric about the other diagonal $x+y=1$. A similar reasoning as above justifies this: 

Define $\sigma\in S_n$ by $\sigma(i):=n+1-\pi^{-1}(n+1-\pi(i))$ and note that if $\pi$ converges to a probability on $[0,1]^2$ with density $g_{f,\theta}(x,y)$, then $\sigma$ converges to a probability on $[0,1]^2$ with density $g_{f,\theta}(1-y,1-x)$. But since
$$\sum_{i=1}^ni\pi(i)=\sum_{i=1}^n(n+1-i)(n+1-\pi(i))=\sum_{i=1}^ni\sigma(i),$$ it follows that  under $\Q_{n,f,\theta}$ the distribution of $\pi$ is same as the distribution of $\sigma$.  Thus $\sigma$ has limiting  density $g_{f,\theta}(x,y)$ as well, which implies $g_{f,\theta}(x,y)=g_{f,\theta}(1-y,1-x)$, and so $g_{f,\theta}$ is symmetric about the line $x+y=1$.

To compare how close the empirical measure $\nu_\pi$ is to the limit, a random permutation $\pi$ of size $n=10000$ is drawn from $\Q_{n,f,\theta}$  via MCMC. The algorithm used to simulate from this model is adopted from \cite{AD}, and is explained below:

\begin{enumerate}
\item
Start with $\pi$  chosen uniformly at random from $S_n$.

\item
Given $\pi$, simulate $\{U_i\}_{i=1}^n$ mutually independent with $U_i$ uniform on $[0,e^{(\theta/n^2) i\pi(i)}]$.

\item
Given $U$, let $b_j:=\max\{(n^2/\theta j)\log U_j,1\} $. Then $1\leq b_j\le n$. Choose an index $i_1$ uniformly at random from set $\{j:b_j\le 1\}$, and set $\pi({i_1})=1$. Remove this index from $[n]$ and choose an index $i_2$ uniformly from $\{j:b_j\le 2\}-\{i_1\}$, and set $\sigma({i_2})=2$. In general, having defined $\{i_1,\cdots,i_{l-1}\}$, remove them from $[n]$, and choose $i_l$ uniformly from  $\{j:b_j\le l\}-\{i_1,i_2\cdots i_{l-1}\}$, and set $\pi({i_l})=l$. [That this step can be always carried out completely  was proved in \cite{DGH}.]

\item
Iterate between the steps 2 and 3 till convergence.

\end{enumerate}
 The above iteration is run $10$ times to obtain a single permutation $\pi$, and then the frequency histogram of the 
 points $\{i/n,\pi(i)/n\}_{i=1}^n$ are computed with $k\times k$ bins, where $k=10$. The mesh plot of the  frequency histogram is given in figure \ref{limit_vs_sample}(b).

The pattern of the histogram in Figure \ref{limit_vs_sample}(b) is very similar to the function plotted in Figure \ref{limit_vs_sample}(a), showing that  the probability assigned by the random permutation $\pi$  has a similar pattern as that of the limiting density $g_{f,\theta}(x,y)$. The histogram has been drawn with $k^2$ squares, each of size $.1$ as $k=10$.

Using  theorem \ref{approximate} gives an approximation to $\frac{1}{n}[Z_n(f,\theta)- Z_n(0)]$ as 
$$\frac{\theta}{k^2} \sum_{r,s=1}^kijB_m(r,s)-2\log k-\sum_{r,s=1}^kB_m(r,s)\log B_m(r,s).$$
Figure \ref{usethis} gives a plot of $\theta$ versus $\lim_{n\rightarrow\infty}\frac{1}{n}[ Z_n(f,\theta)- Z_n(0)]$, where the limiting value is estimated using the above  approximation. For this plot $k$ has been chosen to be $100$, and the range of $\theta$ has been taken to be $[-500,500]$. The number of iterations for the convergence of the iterative algorithm for each $\theta$ has been taken as $20$.
 \begin{figure}[htbp]
\centering
\includegraphics[height=2.7in,width=3.5in]{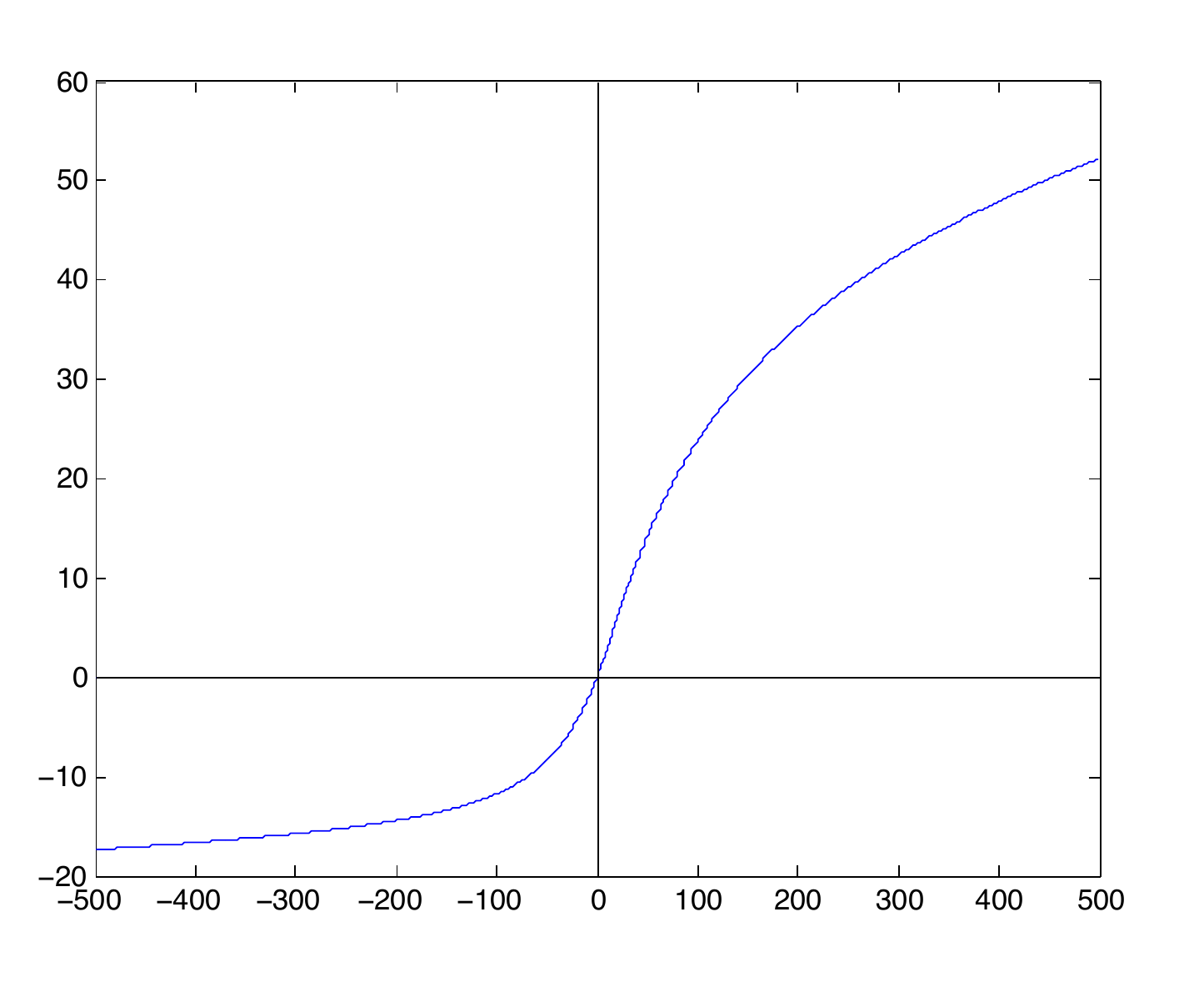}  
\caption{Plot of $\theta$ versus $Z(f,\theta)$ for rank correlation model.}
\centering
\label{usethis}
\end{figure}
The curve passes through $(0,0)$,  and goes to $\pm \infty$ as $\theta$ goes to $\pm \infty$, as expected.

The above method can be used to approximate the limiting log normalizing constant for any model of permutations  described in the setting Theorem \ref{l1}.

\section{Analysis of the 1970 draft lottery data}\label{sec:draft}

This section analyses the 1970 draft lottery data using the methods  developed in this paper. The data for this lottery is taken from \url{http://www.sss.gov/LOTTER8.HTM}. This lottery was used to determine the relative order in which male U.S. citizens born between 1944-1950 will join the army, based on their birthdays. As an example, September $14^{th}$ was the first chosen day, which means that people born on this date had to join first.
Assume that the $366$ days of the year are chronologically numbered,  i.e. January 1 is day 1, and December 31 is day 366. Then the data can be represented in the form of a permutation of size $366$, where $\pi(i)$ represents the $i^{th}$ day chosen in the lottery.  The lottery was carried out in a somewhat flawed manner as follows:
\\

$366$ capsules were made, one for each day of the year. The January capsules were put in a box first, and then mixed among themselves. The February capsules were then put in the box, and the capsules for the first two months were mixed. This was carried on until the December capsules were put in the box, and all the capsules were mixed. As a result of this mixing, the January capsules were mixed 12 times, the February capsules were mixed 11 times, and the December capsules were mixed just once. As a result, most of the capsules for the latter months stayed near the top, and ended up being drawn early in the lottery. The resulting  permutation $\pi$ thus seems to have a bias towards the permutation $$(366,365,\cdots,1),$$ and so the permutation $\tau=367-\pi$ should be biased towards the identity.

Thus the question of interest is to test whether the permutation $\tau$   is chosen uniformly at random from $S_{366}$, and the alternative hypothesis is that $\tau$ has a bias towards the identity permutation.
For $\tau\in S_n$ with $n=366$,  one can construct the histogram of the points $$\Big\{\Big(\frac{i}{n},\frac{\tau(i)}{n}\Big),1\le i\le n\Big\}.$$  If $\tau$ is indeed drawn from the uniform distribution on $S_n$, then then this histogram should be close to the uniform distribution on the unit square. The bivariate histogram is drawn with $10\times 10$ bins in figure \ref{hist_compare}(a). 
To compare this with the uniform distribution on $S_n$, a uniformly random permutation $\sigma$ is chosen from $S_n$, and the corresponding histogram is drawn in figure \ref{hist_compare}(b) with the same the number of bins as above.
From figure \ref{hist_compare} it seems that the heights of the bins in the second picture are a bit more uniform than the first.

\begin{figure*}[h]
\centering
\begin{minipage}[c]{0.5\textwidth}
\centering
\includegraphics[width=3.5in]
    {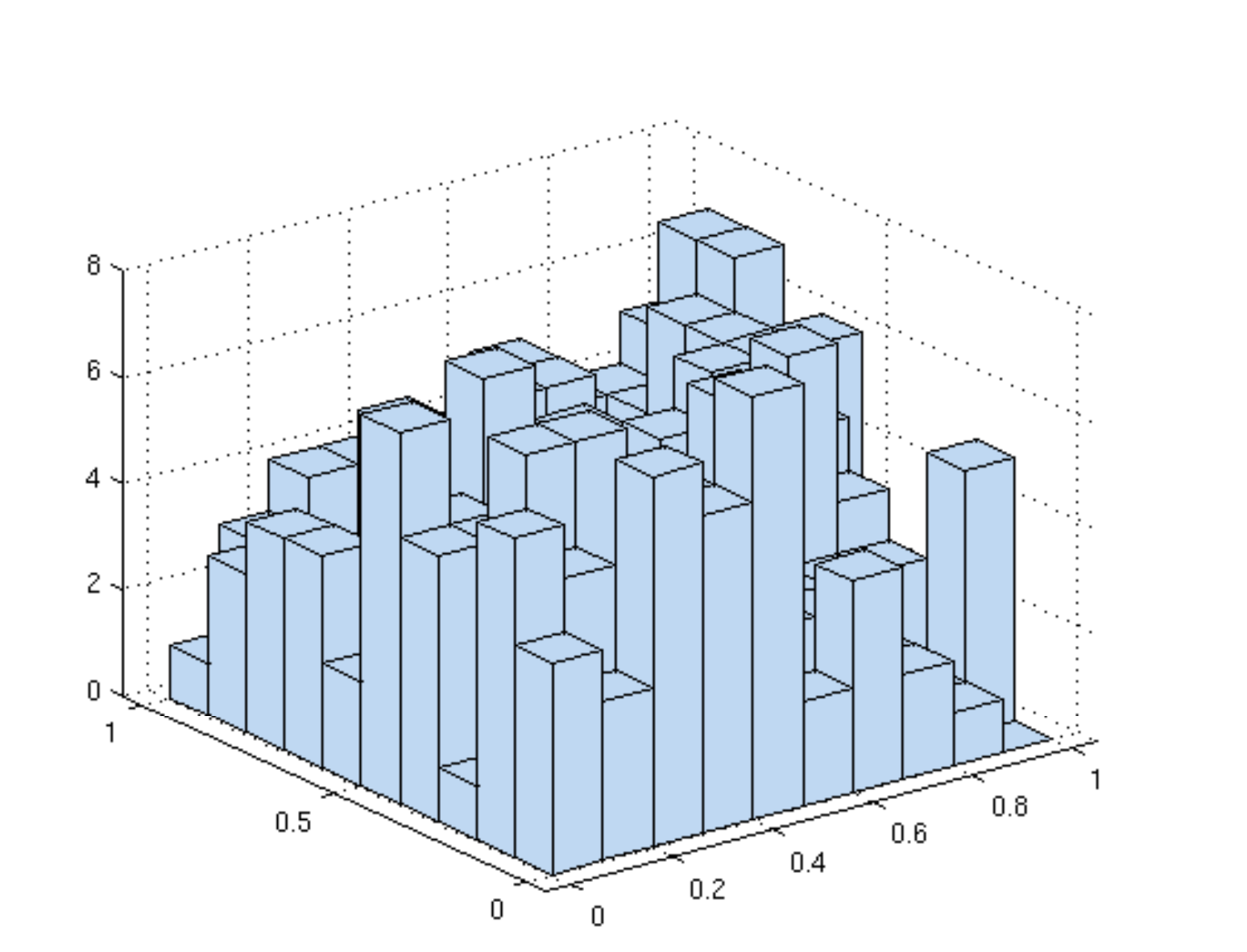}\\
\small{(a)}    
\end{minipage}
\begin{minipage}[c]{0.49\textwidth}
\centering
\includegraphics[width=3.5in]
    {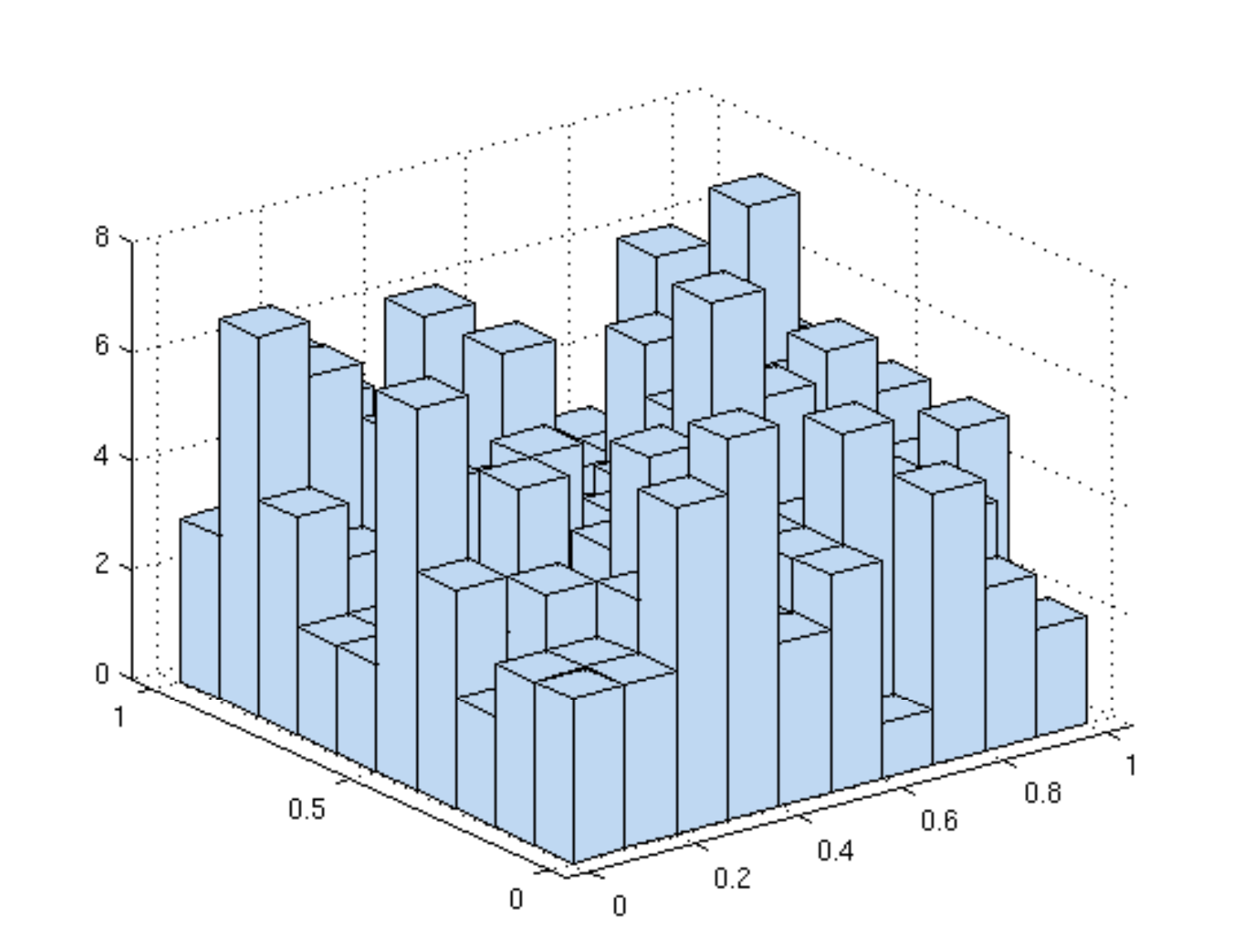}\\
\small{(b)}   
\end{minipage}
\caption{Bivariate histogram of the points $\frac{1}{366}\{(i,\tau(i)),1\le i\le 366\}$ with $10\times 10$ bins where $\tau$ is 
\small{(a) reverse permutation of draft lottery, (b) a random permutation chosen uniformly.}}
\label{hist_compare}
\end{figure*}

%
%
%

 If $\tau$ is indeed uniform, then the statistic
 $\frac{1}{n^3}\sum_{i=1}^n i\tau(i)$ has a limiting normal distribution with mean $\frac{1}{4}\Big(1+\frac{1}{n}\Big)^2\approx 0.25$ and variance $\frac{1}{144n}\Big(1-\frac{1}{n})\Big(1+\frac{1}{n}\Big)^2\approx \frac{1}{144n}\approx 1.89\times 10^{-5}$ (\cite[Page 116]{D}). The observed value of this statistic is $0.2702$, which clearly falls outside a $99\%$ acceptance region under the null hypothesis. Even  if the normal approximation is not believed, by Chebyshev's inequality one has 
 $$\P_{\theta=0}\Big(\frac{1}{366^3}\sum_{i=1}^{366}i\tau(i)\ge 0.2702\Big)\le \frac{1.89\times 10^{-5}}{.0502^2}\approx 0.0075,$$
 which suggests very strong evidence against the null hypothesis.
 \\
 
  If $\tau$ is assumed to be generated from  the model $$\Q_{n,f,\theta}(\tau)=e^{\theta/n^3\sum_{i=1}^ni\tau(i)-Z_n(f,\theta)}$$
 where $f(x,y)=xy$, the  test used above is the most powerful test (in the sense of NP Lemma) for testing  $\theta=0$ versus $\theta>0$. Since the null is rejected, 
 it might be of interest to see if there is another value of $\theta$ for which the model better fits the data.  To investigate this,  the value of $\theta$ is estimated using the estimators $\hat{\theta}_{LD}$ of Corollary \ref{thm:ldmle} and $\hat{\theta}_{PL}$ of  theorem \ref{est}. By a direct computation it turns out that $\hat{\theta}_{PL}=2.92$. To compute $\hat{\theta}_{LD}$ requires estimating the limiting log normalizing constant, for which one needs to carry out the IPFP algorithm of theorem \ref{approximate}. The grid size chosen for computing $\hat{\theta}_{LD}$ is $1000\times 1000$. It follows from the proof of theorem \ref{approximate} that the error in approximating the limiting log partition function $Z(f,\theta)$ by a $k$ step approximation $W_k(f,\theta)$ is bounded by $|\theta|\epsilon_k$, where 
 $$\epsilon_k=\sup_{|x_1-x_2|\le 1/k,|y_1-y_2|\le 1/k}|f(x_1,y_1)-f(x_2,y_2)|\le \frac{2}{k}.$$ Thus a choice of $k=1000$ should ensure that the limiting log partition function is correct upto the first two decimal places, assuming the run time $m$ is large.
Larger values of $k$ will increase accuracy of the estimate, at the cost of speed of computation. For each value of $\theta$ the IPFP algorithm is run $m=200$ times. The estimate $\hat{\theta}_{LD}$ turns out to be $2.96$, which is close to the pseudo-likelihood estimate.
To compare the relative performance of the two estimators $\hat{\theta}_{PL}$ and $\hat{\theta}_{LD}$,  a sample of $1000$ values is drawn from this model for $\theta=2.92$ and $\theta=2.96$, and the histogram of the statistic $n^{-3}\sum_{i=1}^ni\tau(i)$ is plotted side by side in figure \ref{hist_for_spearman} with $25$ bins. The observed value from the draft lottery data is $.2702$, represented by the green line. From figure \ref{hist_for_spearman} it is clear that both estimates give a good fit to the observed data. 
\begin{figure*}[h]
\centering
\includegraphics[width=3in]{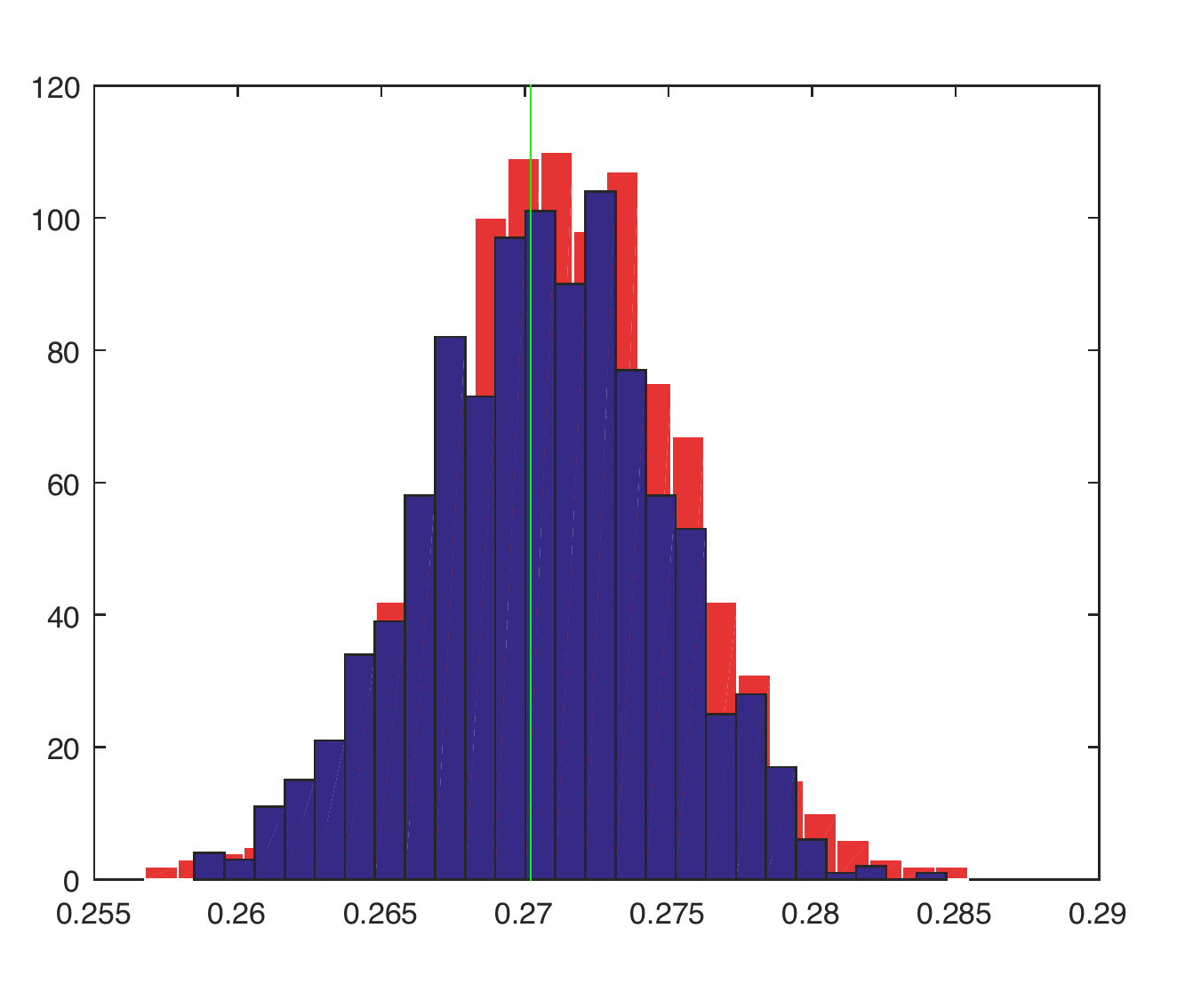}\\
\caption{\small{ Histogram of the statistic $366^{-3}\sum_{i=1}^{366}i\tau(i)$ with $1000$ independent draws grouped into $25$ bins, where $\tau$ is a random permutation from Spearman's rank correlation model with
\small{(a) $\theta=2.92$ in blue (Pseudo-likelihood), (b) $\theta=2.96$ in red (LD-MLE). The green line at $0.2702$ is obtained  when $\tau$ is the reverse permutation of Draft Lottery data.}}}
\label{hist_for_spearman}
\end{figure*}


Finally, to test whether these values of $\theta$ gives a good fit to the given data, an independent random permutation $\hat{\tau}$ is drawn from this model with $\theta=2.92$. The same auxiliary variable algorithm of Andersen-Diaconis from the previous section is used to draw the sample.
The histogram of $\hat{\tau}$ is given below in figure \ref{hist_compare_fit}(b) with $10\times 10$ bins, along side the histogram for the observed permutation $\tau$ in \ref{hist_compare_fit}(a).

\begin{figure*}[h]
\centering
\begin{minipage}[c]{0.5\textwidth}
\centering
\includegraphics[width=3.5in]
    {draft}\\
\small{(a)}    
\end{minipage}
\begin{minipage}[c]{0.49\textwidth}
\centering
\includegraphics[width=3.5in]
    {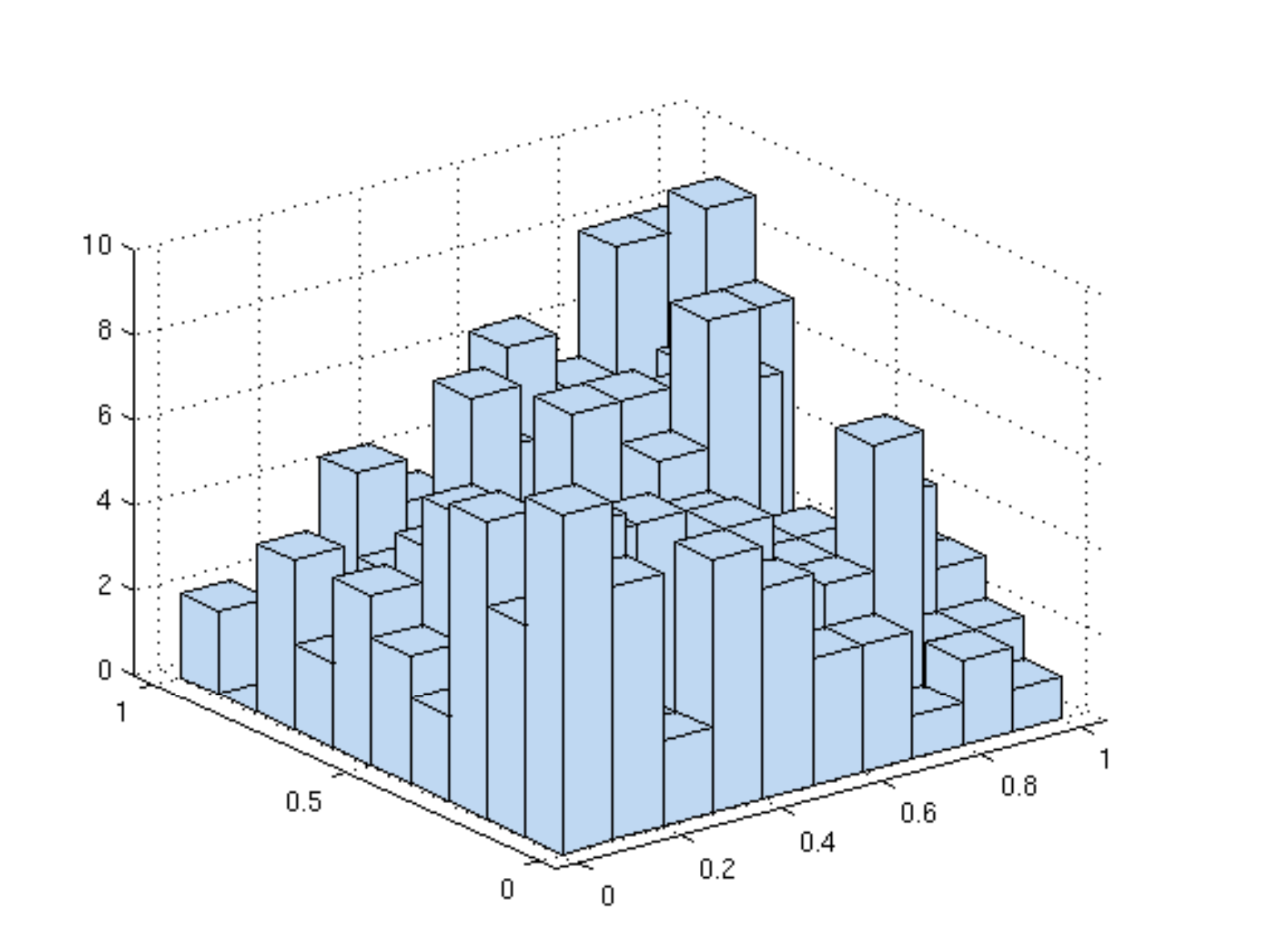}\\
\small{(b)}   
\end{minipage}
\caption{Bivariate histogram of the points $\frac{1}{366}\{(i,\tau(i)),1\le i\le 366\}$ with $10\times 10$ bins where $\tau$ is 
\small{(a) reverse permutation of draft lottery, (b) a random permutation chosen from Spearman's model with $\theta=2.92$.}}
\label{hist_compare_fit}
\end{figure*}

The bivariate histogram of  the points $(i/n,\tau(i)/n)_{i=1}^{n}$ for the observed permutation $\tau$ and the points $(i/n,\hat{\tau}(i)/n)_{i=1}^{n}$ for the simulated permutation $\hat{\tau}$ is drawn in figure \ref{hist_compare_fit}. This seems to be a better match than the histograms for $\tau$ and $\sigma$ in figure \ref{hist_compare_fit}, where $\sigma$ was a permutation drawn uniformly at random. This agrees with the observation made in \cite{F} that the observed permutation does not seem uniformly random.


\section{Appendix: Proofs of main results}\label{appen:1}

\subsection{Permutation limits}

%
%

The concept of permutation limits was introduced in \cite{HKMRS} in 2011, and was  motivated from graph limit theory. For a brief exposition of the theory of graph limits refer to Lovasz \cite{LB}.  The central idea in permutation limit theory is that any permutation can be thought of as a  probability measures $\sM$ on $[0, 1]^2$ with uniform marginals. For any $\pi\in S_n$, define a probability measure $\mu_{\pi}\in \sM$
as $d\mu_{\pi}:=f_{\pi}(x,y) dxdy$, where $f_{\pi}(x,y)=n\pmb 1\{(x,y):\pi(\lfloor n x\rfloor)=\lfloor ny\rfloor\}$ is the density of $\mu_{\pi}$ with respect to Lebesgue measure. An intuitive definition of $\mu_\pi$ is as follows:
\\

 Partition $[0, 1]^2$ into $n^2$ squares of side length $1/n$, and define $f_{\pi}(x, y)=n$ for all $(x, y)$ in the $(i, j)$-th square if $\pi(i)=j$ and 0 otherwise. As an example, the measure $\mu_\pi$ corresponding to the permutation $\pi=(1,3,2)$ has the density of figure \ref{perm_eg}.
 \begin{figure*}[htbp]
\centering
\includegraphics[height=1.5in,width=1.5in]{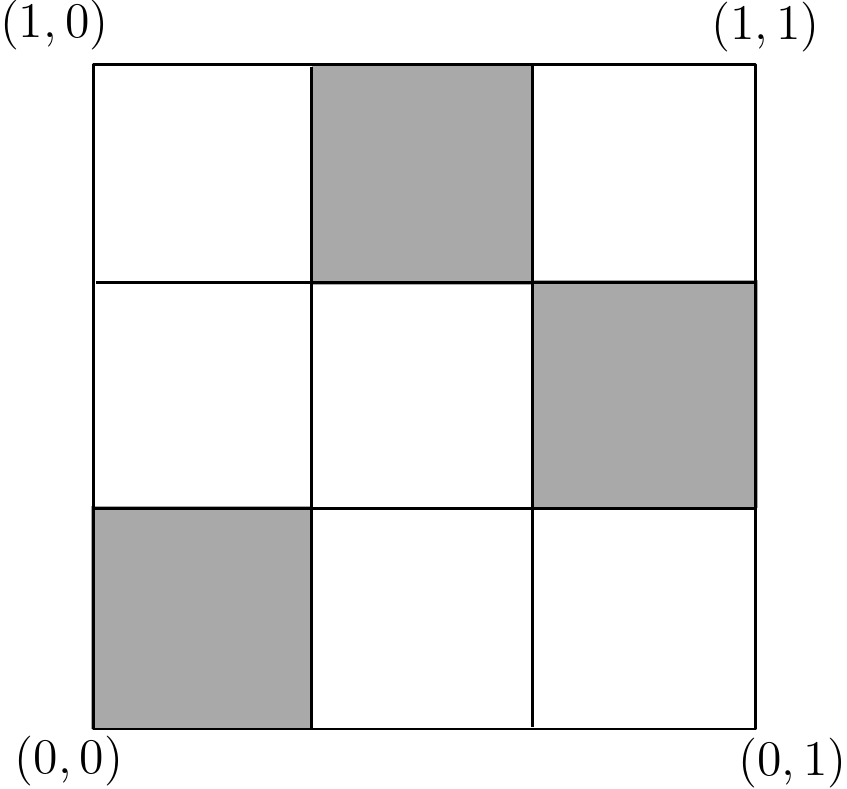} 
\caption{Measure representation for the permutation $(1,3,2)$. Here the shaded region has density $3$, and the white region has density $0$.} 
\centering
\label{perm_eg}
\end{figure*}
Here the shaded region has density $3$, and the white region has $0$ density.

 A sequence of permutations $\pi_n\in S_n$ is said to converge to a measure $\mu\in \cM$, if the corresponding sequence of probability measures $\mu_{\pi_n}$ converge weakly to $\mu$. As an example if $\pi_n$ is uniformly distributed on $S_n$, then $\pi_n$ converges to Lebesgue measure on $[0,1]^2$. If $\pi_n=(1,2,\cdots,n)$ is the identity permutation on $S_n$, then $\pi_n$ converges to a measure which is uniform on the diagonal $x=y$. Similarly if $\pi_n=(n,n-1,\cdots,1)$ is the reverse permutation, then $\pi_n$ converges to the uniform measures on the diagonal $x+y=1$. For non trivial limits that can arise as permutation limits, refer to Theorem \ref{l1} and Proposition \ref{mallow}.

\subsection{The large deviation principle}

Given a permutation $\pi$, the previous subsection defined a measure $\mu_\pi$ on the unit square. Also recall part (b) of theorem \ref{l1} which, given a permutation $\pi\in S_n$, defines a measure $$\nu_\pi=\frac{1}{n}\sum_{i=1}^n\delta_{(i/n,\pi(i)/n)}$$ on the unit square.
Both marginals of $\nu_\pi$  are discrete uniform on  the set $\{(i/n),i\in 1,2,\cdots,n\}$. Since the marginals are not uniform on $[0,1]$, $\nu_\pi$ is not an element of $\sM$, but any weak limit of the sequence $\nu_{\pi_n}$ is in $\sM$ if the size of the permutation  goes to $\infty$.
If the size of the permutation $\pi$ is large, the two measure $\mu_\pi$ and $\nu_\pi$ are close in the weak topology. To see this, let $\pi\in S_n$, and let $F_{\mu_\pi}$ and $F_{\nu_\pi}$ represent the bivariate distribution functions of $\mu_\pi$ and $\nu_\pi$ respectively.   Then it follows that
$$d_\infty(\mu_\pi,\nu_\pi):=\sup_{0\le x,y\le 1}|F_{\mu_\pi}(x,y)-F_{\nu_\pi}(x,y)|\le \frac{2}{n}.$$
To see this note that both $\mu_\pi$ and $\nu_\pi$ can be defined by partitioning the unit square into $n^2$ boxes, such that exactly $n$ boxes receive a mass of $1/n$. Also the choice of the $n$ boxes is such that every row and every column will have exactly one box of positive mass. Thus  any vertical line through $x$ can intersection exactly one box in this partition which has positive probability, and so the above difference can be at most $1/n+1/n$.

 The main tool for proving the results of this paper is a large deviation principle for $\mu_{\pi}$ with respect to weak convergence on $\sM$ where $\pi \sim\P_n$, the uniform probability measure on $S_n$. This result is stated below.

\begin{thm}\label{ldp}
If  $\pi\sim \mathbb{P}_n$, the uniform measure on $S_n$,  the sequence of probability measures $\mu_\pi$  satisfies a large deviation principle on $\cM$ with the good rate function $D(\mu||u)$, 
where $u$ is the uniform measure on  $[0,1]^2$. More precisely, for   any set $A\subset \cM$ one has
\[-\inf_{\mu\in A^o}D(\mu||u)\leq \liminf_{n\rightarrow\infty}\frac{1}{n}\log\mathbb{P}_n(A)\leq \limsup_{n\rightarrow\infty}\frac{1}{n}\log \mathbb{P}_n(A)\leq -\inf_{\mu\in \overline{A}}D(\mu||u),\] where  $A^o$ and $\overline{A}$  denotes the  interior and closure of $A$ respectively.
\end{thm}

The following proposition derives  the large deviation of   $\nu_\pi$ from that of $\mu_\pi$.

\begin{ppn}\label{aaa}
If $\pi\sim \mathbb{P}_n$, the uniform probability measure on $S_n$, the sequence of probability measures $\nu_\pi$ satisfy a large deviation principle on the  space of probability measures on $[0,1]^2$ with respect to the weak topology, with the good rate function $\overline{I}(.)$  given by 
\begin{align*}
\overline{I}(\mu):=&D(\mu||u),\text{ if }\mu\in\sM,\\
:=&\infty \text{ otherwise.}
\end{align*}
%

\end{ppn}

\begin{proof}

Since the set of all probability measures $[0,1]^2$ is compact, the set $\sM$ is compact as well. An application of  \cite[Lemma 4.1.5 (a)]{DZ} and the large deviation result for $\mu_\pi$ (Theorem \ref{ldp}) gives that under $\mathbb{P}_n$, the sequence $\mu_\pi$ satisfies a large deviation principle on the space of probability measures on $[0,1]^2$ with the rate function $\overline{I}$.  Since the two sequences $\mu_\pi$ and $\nu_\pi$ are close in the $d_\infty(.,.)$, by \cite[Theorem 4.2.13]{DZ} they have the same large deviation. 

\end{proof}

Theorem \ref{l1} now follows from Proposition \ref{aaa} as follows.

\begin{proof}[Proof of Theorem \ref{l1}]
\begin{enumerate}[(a)]
\item

Note that 
\[e^{Z_n(f,\theta)-Z_n(0)}=\frac{1}{n!}\sum_{\pi\in S_n}e^{\theta \sum_{i=1}^nf(i/n,\pi(i)/n)}=\mathbb{E}_{\mathbb{P}_n}e^{n\theta \nu_\pi[f]},\]
where $Z_n(0)=\log n!$, and $\mu[f]=\int_{[0,1]^2} fd\mu$  denotes the mean of $f$ with respect to $\mu$. Since the function $\mu\mapsto \theta \mu[f]$ is bounded and continuous, an application of Varadhan's Lemma \cite[Theorem 4.3.1]{DZ} along with the large deviation of $\nu_\pi$  gives the desired conclusion.
\\

\item

The function $\mu\mapsto \theta \mu[f]-D(\mu||u)$ is strictly concave (on the set where it is finite) and upper semi continuous on the compact set $\sM$, and so the global maximum is attained at a unique $\mu_{f,\theta}\in \sM$. To show the weak convergence of $\nu_\pi$ fix  an open set $U$ containing $\mu_{f,\theta}$, define a function $T:\sM\mapsto [-\infty,\infty)$ by $$T(\mu)=\theta\mu[f]\text{ if }\mu\in U^c, \quad -\infty\text{ otherwise }.$$
Then
$$\frac{1}{n}\log \Q_{n,f,\theta}(\nu_\pi\in U^c)=\frac{1}{n}\log \E_{\P_n}e^{nT(\nu_\pi)}-\frac{1}{n} Z_n(f,\theta).$$
Since $T$ is upper semi continuous and bounded above, \cite[Equation 4.3.2]{DZ}  holds trivially and so by \cite[Lemma 4.3.6]{DZ} along with the large deviation result for $\nu_\pi$ one has 
$$\limsup_{n\rightarrow\infty}\frac{1}{n}\log\E_{\P_n}e^{nT(\nu_\pi)}\le \sup_{\mu\in U^c\cap \sM}\{\theta\mu[f]-D(\mu||u)\}.$$
This, along with part (a) gives
$$\limsup_{n\rightarrow\infty}\frac{1}{n}\log \Q_{n,f,\theta}(\nu_\pi\in U^c)\le  \sup_{\mu\in U^c\cap \sM}\{\theta \mu[f]-D(\mu||u)\}-\sup_{\mu\in\sM}\{\theta \mu[f]-D(\mu||u)\}.$$
  The quantity on the right hand side above is negative as the infimum over the compact set $U^c\cap \sM$ is attained, and the global minimizer $\mu_{f,\theta}$ is not in $U^c$ by choice.
   This proves that $\Q_{n,f,\theta})(\nu_\pi\in U^c)$ decays to $0$  at an exponential rate, which in particular implies that $\{\nu_\pi\}$ converges to $\mu_{f,\theta}$ weakly in probability.

\item
Since $\theta f(.)$ is integrable with respect to $du$, by \cite[Corollary 3.2]{Cs}  there exists  functions $a_{f,\theta}(.),b_{f,\theta}(.):\in L^1[0,1]$ such that \[d\mu_{a,b}=g_{a,b}dxdy:=e^{\theta f(x,y)+a_{f,\theta}(x)+b_{f,\theta}(y)}dxdy\in \sM.\] 
The proof that $\mu_{a,b}=\mu_{f,\theta}$ is by way of contradiction. Suppose this is not true. Since $\mu_{f,\theta}$ is the unique global minimizer  of $I_{f,\theta}(\mu):=D(\mu||u)-\theta\mu[f]$, setting
 \[h(\alpha):= I_{f,\theta}((1-\alpha)\mu_{a,b}+\alpha\mu_{f,\theta})\] it must be that  $h(\alpha)$  has a global minima at $\alpha=1$. Also 
 \[I_{f,\theta}(\mu_{f,\theta})\leq I_{f,\theta}(u)=-\theta u(f)<\infty,\]
   which forces $D(\mu_{f,\theta}||u)<\infty$. Thus letting $\phi_{f,\theta}:=\frac{d\mu_{f,\theta}}{du}$ gives
 \begin{align*}
 h'(0)=&\int_T(\phi_{f,\theta}(x,y)-g_{a,b}(x,y))(\log g_{a,b}(x,y)-\theta f(x,y))du\\
 =&\int_T(\phi_{f,\theta}(x,y)-g_{a,b}(x,y))( a_{f,\theta}(x)+ b_{f,\theta}(y))du\\
 =&\mathbb{E}_{\mu_{f,\theta}}[ a_{f,\theta}(X)+ b_{f,\theta}(Y)]-\mathbb{E}_{\mu_{a,b}}[ a_{f,\theta}(X)+ b_{f,\theta}(Y)]=0,
 \end{align*}
where the last equality follows from the fact that both $\mu_{f,\theta}$ and $\mu_{a,b}$ have the same uniform marginals.
But $h$ is convex, which forces that $\alpha=0$ is also a global minima of $h(.)$. Thus $h(0)=h(1)$, a contradiction to the uniqueness of $\arg\max_{\mu\in\sM}\{\theta \mu[f]-D(\mu||u)\}$ proved in part (b). Thus it must be that
  $$d\mu_{f,\theta}=d\mu_{a,b}=e^{\theta f(x,y)+a_{f,\theta}(x)+b_{f,\theta}(y)}dxdy.$$ Finally, the almost sure uniqueness of $a_{f,\theta}(.)$ and $b_{f,\theta}(.)$ follows from the uniqueness of the optimizing measure $\mu_{f,\theta}$.
The last claim of part (c) then follows from part (a) by a simple calculation.

\item
Since $\nu_\pi$ converges in probability to $\mu_{f,\theta}$, it follows by Dominated Convergence theorem that
$$Z_n'(f,\theta)=\E_{\Q_{n,f,\theta}} \frac{1}{n}\sum_{i=1}^nf(i/n,\pi(i)/n)\stackrel{n\rightarrow\infty}{\rightarrow}\mu_{f,\theta}[f].$$
Another application of Dominated Convergence theorem gives that
$$\frac{1}{n}[Z_n(f,\theta)-Z_n(0)]=\int_0^\theta \frac{1}{n}Z_n'(f,t)dt\stackrel{n\rightarrow\infty}{\rightarrow}\int_0^\theta \mu_{f,t}[f],$$
which along with part (a) gives that $Z'(f,\theta)=\mu_{f,\theta}[f]$. 
\\

Since $Z(f,\theta)$ is convex $Z'(f,\theta)$ is non-decreasing. To show that $Z'(f,\theta)$ is strictly increasing, by way of contradiction let $\theta_1\ne \theta_2$ be such that $Z'(f,\theta_1)=Z'(f,\theta_2)$ for some $\theta_1\ne \theta_2$, which implies $\mu_{f,\theta_1}[f]=\mu_{f,\theta_2}[f]$. The optimality  of $\mu_{f,\theta_1}$ gives $$\theta_1\mu_{f,\theta_1}[f]-D(\mu_{f,\theta_1}||u)\ge  \theta_1\mu_{f,\theta_2}[f]-D(\mu_{f,\theta_2}||u),$$
which implies $D(\mu_{f,\theta_1}||u)\le D(\mu_{f,\theta_2}||u)$. By symmetry  $D(\mu_{f,\theta_1}||u)=D(\mu_{f,\theta_2}||u)$, and so $\theta_1 \mu_{f,\theta_1}[f]-D(\mu_{f,\theta_1}||u)=\theta_1\mu_{f,\theta_2}[f]-D(\mu_{f,\theta_2}||u)$. This implies $\mu_{f,\theta_1}=\mu_{f,\theta_2}$ by the uniqueness of theorem \ref{l1} part (b). By the form of the optimizer proved in theorem \ref{l1} part (c) one has
$$e^{\theta_1 f(x,y)+a_{f,\theta_1}(x)+b_{f,\theta_1}(y)}=e^{\theta_2 f(x,y)+a_{f,\theta_2}(x)+b_{f,\theta_2}(y)},$$ which on taking log gives
$f(x,y)=\frac{1}{\theta_1-\theta_2}\Big(a_{f,\theta_2}(x)+b_{f,\theta_2}(y)-a_{f,\theta_1}(x)-b_{f,\theta_1}(y)\Big).$
Integrating with respect to $y$ using the definition of $\cC$ gives 
$a_{f,\theta_1}(x)-a_{f,\theta_2}(x)=\int_0^1 [b_{f,\theta_2}(y)-b_{f,\theta_1}(y)]dy,$ and so $a_{f,\theta_1}(x)-a_{f,\theta_2}(x)$ is a constant. By symmetry $b_{f,\theta_1}(y)-b_{f,\theta_2}(y)$ is a constant as well, and so $f(x,y)$ is constant, a contradiction to the assumption that $f\in \cC$.  
\\

Finally to show continuity of $Z'(f,\theta)$, let $\theta_k$ be a sequence of reals converging to $\theta$. Since sequence of measures $\mu_{f,\theta_k}\in \cM$ is tight, let $\mu$ be any limit point of this sequence.  Then by continuity of $Z(f,.)$ and lower semi continuity of $D(.||u)$ one has $$Z(f,\theta)=\limsup_{k\rightarrow\infty}Z(f,\theta_k)=\limsup_{k\rightarrow\infty}\{\theta_k\mu_{f,\theta_k}[f]-D(\mu_{f,\theta_k}||u)\}\le \theta \mu[f]-D(\mu||u).$$
Since $Z(f,\theta)=\sup_{\mu\in \cM}\{\theta \mu[f]-D(\mu||\theta)\}$ and the supremum is attained uniquely at $\mu_{f,\theta}$ it follows that $\mu=\mu_{f,\theta}$, and so the sequence $\mu_{f,\theta_k}$ converge weakly to $\mu_{f,\theta}$. But this readily implies
$$Z'(f,\theta_k)=\mu_{f,\theta_k}[f]\stackrel{k\rightarrow\infty}{\rightarrow}\mu_{f,\theta}[f]=Z'(f,\theta),$$
and so $Z'(f,.)$ is continuous, thus completing the proof of the theorem.
\end{enumerate}
\end{proof}

\begin{proof}[Proof of Corollary \ref{thm:ldmle}]
\begin{enumerate}[(a)]
\item
Since $\frac{1}{n}\sum_{i=1}^nf(i/n,\pi(i)/n)=\nu_\pi[f]$ and $\nu_f$ converges weakly to $\mu_{f,\theta}$ by Theorem \ref{l1}, the desired conclusion follows.

\item
Fixing $\delta>0$ by part (a) one has
$$LD_n(\pi,\theta_0+\delta)\stackrel{p}{\rightarrow}Z'(\theta_0)-Z'(\theta_0+\delta)<0,\quad LD_n(\pi,\theta_0-\delta)\stackrel{p}{\rightarrow}Z'(\theta_0)-Z'(\theta_0-\delta)>0,$$
and so by continuity and strict monotonicity of $Z'(f,\theta)$ from part (d) of Theorem \ref{l1} it follows that with probability tending to $1$ there exists a unique root $\hat{\theta}_{LD}$ of the equation $LD_n(\pi,\theta)=0$, and $|\hat{\theta}_{LD}-\theta_0|\le \delta$. This proofs the consistency of $\hat{\theta}_{LD}$. The proof of consistency of $\hat{\theta}_{ML}$ follows verbatim by replacing $LD_n(\pi,\theta)$ with $ML_n(\pi,\theta)$.

\item
Since $\hat{\theta}_{LD}$ converges to $\theta_0$ under $\Q_{n,f,\theta_0}$ and to $\theta_1$ under $\Q_{n,f,\theta_1}$ the conclusion follows.

\end{enumerate}
\end{proof}
The following definition will be used in the proof of theorem \ref{approximate}.
\begin{defn}\label{def-1}
For $k\in\mathbb{N}$, partition $[0,1]^2$ into $k^2$  squares $\{T_{rs}\}_{r,s=1}^k$ of length $1/k$, with 
\begin{align*}
T_{rs}:=&\Big\{(x,y)\in T:\lceil kx \rceil =r,\lceil ky \rceil =s\Big\}\text{ for }2\le r,s,\le k,\\
T_{1s}:=&\Big\{(x,y)\in T:\lceil kx \rceil \le 1,\lceil ky \rceil =s\Big\}\text{ for }2\le s\le k,\\
T_{r1}:=&\Big\{(x,y)\in T:\lceil kx \rceil \le 1,\lceil ky \rceil =s\Big\}\text{ for }2\le r\le k,\\
T_{11}:=&\Big\{(x,y)\in T:\lceil kx \rceil \le 1,\lceil ky \rceil \le 1\Big\}.
\end{align*}
Also define the  $k\times k$ matrix $M(\pi)$ by \begin{align}\label{eq:define_M}
M_{rs}(\pi):=\sum_{i=1}^n1\{(i/n,\pi(i)/n)\in T_{rs}\}=n\nu_\pi(T_{rs}).
\end{align}
The definition ensures that $T_{rs}$ is a disjoint partition of $[0,1]^2$, and so sum of the elements of $M(\pi)$ is $n$.  It should be noted that all the sets $T_{rs}$ above are $\mu$ continuity sets for any $\mu\in \sM$. This readily follows from noting that the boundary of $T_{rs}$ is contained in $$\Big\{(x,y):x=\frac{r}{k}\Big\}\cup \Big\{(x,y):x=\frac{r-1}{k}\Big\}\cup \Big\{(x,y):y=\frac{s}{k}\Big\}\cup \Big\{(x,y):y=\frac{s-1}{k}\Big\},$$
which has probability $0$ under any $\mu\in \sM$, as $\mu$ has uniform marginals.
\end{defn}
\begin{defn}\label{def:pA}

For any $k\times k$ matrix $A$  two probability distributions $p_A,\widetilde{p}_A$ on the unit square are defined below:

The measure $p_A$ is  a discrete distribution with the p.m.f. $p_A(r/k,s/k)=A_{rs}$ for $1\le r,s\le k$. The measure 
$\widetilde{p}_A$ has a density with respect to Lebesgue measure given by $p_A(x,y)=:k^2A_{rs}$ for $x,y\in T_{rs},1\le r,s\le k$.  
The assumption $A\in \cM_k$ ensures that both $p_A,\widetilde{p}_A$ are probability measures, and further  $\widetilde{p}_A\in \cM$, i.e. it has uniform marginals.
\end{defn}
\begin{proof}[Proof of Theorem \ref{approximate}]
\begin{enumerate}[(a)]
\item
On applying \cite[Theorem 1,2]{Sink} one gets the conclusion that $B_m$ converges to a matrix $A_{k,\theta}\in \cM_k$ of the form $\Lambda_1B_0\Lambda_2$, where $\Lambda_1$ and $\Lambda_2$ are diagonal matrices.

\item

To begin note that
$$\theta\sum_{r,s=1}^kf(r/k,s/k)A(r,s)-2\log k-\sum_{r,s=1}^kA(r,s)\log A(r,s)=\theta p_A[f]-D(p_A||p_{U_k}),$$
where $U_k\in \cM_k$ is  defined by $U_k(r,s):=\frac{1}{k^2}$. By compactness of $\cM_k$ and strong concavity of $A\mapsto \theta p_A[f]-D(p_A||p_{U_k})$ there is a unique maximizer in $\cM_k$, and by \cite[Theorem 3.1]{Cs} it follows that this maximizer is of the form $D_1B_0D_2$ for some diagonal matrices $D_1,D_2$.   Since both $\Lambda_1B_0\Lambda_2$ and $D_1B_0D_2$ are in $\cM_k$, by the uniqueness of \cite[Theorem 1]{Sink} one has $D_1B_0D_2=\Lambda_1B_0\Lambda_2=A_{k,\theta}$, thus completing the proof of part (b).

\item
 Since the function $(\theta,A)\mapsto \theta p_A[f]-D(p_A||p_{U_k})$ from $\R\times \cM_k$ to $[-\infty,\infty)$ is linear in $\theta$, and has a unique maximizer $A_{k,\theta}$ in $A$  for every $\theta$ fixed,
 the conclusion follows on applying Danskin's theorem \cite[B.5]{Bert}.
 
\item
Since $\mu\mapsto \{\mu(T_{rs})\}_{r,s=1}^k$ is a continuous map, by theorem \ref{ldp} and \cite[Theorem 4.2.1]{DZ} the matrix $\frac{1}{n}M(\pi)$ satisfies a large deviation principle on the set of $k\times k$ matrices with the good rate function $$I_k(A):=\inf_{\mu\in \cM:\mu(T_{rs})=A_{rs},1\le r,s\le k}D(\mu||u)$$ if $A\in \cM_k$, and $+\infty$ otherwise. By \cite[Theorem 3.1]{Cs} the maximum is achieved at $\mu=\widetilde{p}_A$, and so
$$I_k(A)=D(\widetilde{p}_A||u)=\sum_{r,s=1}^kA_{rs}\log A_{rs}+2\log k=D(p_A||p_{U_k}).$$ An application of Varadhan's Lemma gives
\begin{align*}
\frac{1}{n}\log \E_{\P_n}e^{\theta \sum_{r,s=1}^k f(r/k,s/k)M_{rs}(\pi)}
=&\frac{1}{n}\log \E_{\P_n}e^{n\theta \sum_{r,s=1}^k f(r/k,s/k)\nu_\pi(T_{rs})}\\
\stackrel{n\rightarrow\infty}{\rightarrow}&\sup_{A\in \cM_k}\{\theta\sum_{r,s=1}^kf(r/k,s/k)A(r,s)-D({p}_A||u)\}=W_k(f,\theta).
\end{align*} Since
$$| \sum_{r,s=1}^k f(r/k,s/k)\nu_\pi(T_{rs})-\frac{1}{n}\sum_{i=1}^nf(i/n,\pi(i)/n)|\le \sup_{|x_1-x_2|\le 1/k,|y_1-y_2|\le 1/k}|f(x_1,y_1)-f(x_2,y_2)|=:\epsilon_k,$$ it follows that
\begin{align*}
|W_k(f,\theta)-Z(f,\theta)|=\Big|\lim_{n\rightarrow\infty}\frac{1}{n}\log \frac{\E_{\P_n}e^{n\theta \sum_{r,s=1}^k f(r/k,s/k)\nu_\pi(T_{rs})}}{\E_{\P_n}e^{n\theta\sum_{i=1}^n f(i/n,\pi(i)/n)}}\Big|\le |\theta|\epsilon_k.
\end{align*}
By continuity of $f$ one has $\epsilon_k\rightarrow 0$, and so $W_k(f,\theta)$ converges to $Z(f,\theta)$. 
\\

To complete the proof assume that $$p_{A_{k,\theta}}\stackrel{w}{\rightarrow}\mu_{f,\theta}.$$

In this case it follows that
$$W_k'(f,\theta)=p_{A_{k,\theta}}[f]\stackrel{k\rightarrow\infty}{\rightarrow}\mu_{f,\theta}[f]=Z'(f,\theta),$$
and so by Dominated Convergence we have
$\lim_{k\rightarrow\infty}W_k(f,\theta)=Z(f,\theta)$. Finally since $$\lim_{m\rightarrow\infty}\sum_{r,s=1}^kB_m(r,s)\log B_m(r,s)=\sum_{r,s=1}^k A_{k,\theta}(r,s)\log A_{k,\theta}(r,s)$$ by part (a), it follows that
$$Z(f,\theta)=\lim_{k\rightarrow\infty}\lim_{m\rightarrow\infty}\{\theta \sum_{r,s=1}^kf(r/k,s/k)B_m(r,s)-2\log k-\sum_{r,s=1}^kB_m(r,s)\log B_m(r,s)\},$$
which is the desired conclusion. 
\\

It thus remains to show that $p_{A_{k,\theta}}$ converges weakly to $\mu_{f,\theta}$ as $k\rightarrow\infty$. Since the set of probability measures on $[0,1]^2$ is compact, the sequence $p_{A_{k,\theta}}$ is tight. If $\mu\neq \mu_{f,\theta}$ be a limit point, then by joint lower semi continuity of $D(.||.)$ one has
$$\limsup_{k\rightarrow\infty}W_k(f,\theta)=\limsup_{k\rightarrow\infty}\{\theta p_{A_{k,\theta}}[f]-D(p_{A_{k,\theta}}||p_{U_k})\}\le\theta \mu[f]-D(\mu||u)<Z(f,\theta).$$
But this is a contradiction to the fact that $W_k(f,\theta)$ converges to $Z(f,\theta)$, and hence $p_{A_{k,\theta}}$ does indeed converges to $\mu_{f,\theta}$. This completes the proof of the theorem.

\end{enumerate}

\end{proof}
%
%

Before proving Theorem \ref{est},  a  general lemma is stated which constructs $\sqrt{n}$ consistent estimates of $\theta$ in permutation models. The idea of this proof is taken from \cite{Ch}.
\begin{lem}\label{general}
Let $\R_{n,\theta}$ be any one parameter family on $S_n$, and let $G_n(\pi,\theta)$ be a function on $S_n\times \R$ which is differentiable in $\theta$.

Suppose  the following two conditions hold:
\begin{enumerate}[(a)]
\item
For every $\theta_0\in \R$ there exists a constant $C=C(\theta_0)$ such that
\begin{align}\label{cnc1}\mathbb{E}_{\mathbb{\R}_{n\theta_0}}G_n(\pi,\theta_0)^2\le C n^3\end{align}

\item
There exists a strictly positive continuous function $\lambda:\R\mapsto\R$  such that
\begin{align}
\label{cnc2}
\lim_{n\rightarrow\infty}\mathbb{R}_{n,\theta_0}(G_n'(\pi,\theta)\leq -n^2 \lambda(\theta),\forall\theta\in \R)=1.
\end{align}

Then the equation $G_n(\pi,\theta)=0$ has a unique root  in $\theta$. Further denoting this unique root by $\hat{\theta}_n$ one has $\sqrt{n}(\hat{\theta}_n-\theta_0)$ is $O_P(1)$  under $\R_{n,\theta_0}$.
\end{enumerate}

\end{lem}
\begin{proof}

Fixing a large positive real $M$ let $A_n$ denote the set \[A_n:=\{\pi\in S_n:|G_n(\pi,\theta_0)|\leq n^{3/2}M,\ G'_n(\pi,\theta)\leq -n^2\lambda(\theta),\theta\in \R\}.\]  
Then for $\pi \in S_n$ one has 
\[G_n(\pi,\theta_0+1)=G_n(\pi,\theta_0)+\int_{\theta_0}^{\theta_0+1}G_n'(\pi,\theta)d\theta\leq n^{3/2}M-n^2 \inf_{\theta\in[\theta_0,\theta_0+1]}\lambda(\theta),\]  which is negative for all large $n$.
 Similarly it can be shown that $G_n(\pi,\theta_0-1)>0$ for $\pi\in A_n$. Also note that  $G_n(\pi,\theta)$ is strictly monotone on $A_n$, and so  by continuity of $\theta\mapsto G_n(\pi,\theta)$ there exists a unique $\hat{\theta}_n$ satisfying $G_n(\pi,\hat{\theta}_n)=0$, and $\theta_0-1 <\hat{\theta}_n<\theta_0+1$. Finally one has
\[n^{3/2}M\geq |G_n(\pi,\theta_0)|=|G_n(\pi,\theta_0)-G_n(\pi,\hat{\theta}_n)|\geq n^2|\int_{\hat{\theta}_n}^{\theta_0}\lambda(\theta)d\theta| \ge \Big[\inf_{|\theta-\theta_0|\le 1}\lambda(\theta)\Big]|\hat{\theta}_n-\theta_0|,\] and so $\sqrt{n}|\hat{\theta}_n-\theta_0|\leq KM$, where  $K:=[\inf_{|\theta-\theta_0|\leq 1}\lambda(\theta)]^{-1}<\infty$. Thus using \eqref{cnc1} and \eqref{cnc2} gives
\begin{align*}
\limsup_{n\rightarrow\infty}\R_{n,\theta_0}(|\hat{\theta}_n-\theta_0|>KM)\le& \limsup_{n\rightarrow\infty}\R_{n,\theta}(|G_n(\pi,\theta_0)|\ge M n^{3/2})\\\le &\limsup_{n\rightarrow\infty}\frac{1}{M^2n^3}\E_{\R_{n,\theta_0}}G_n(\pi,\theta_0)^2\le \frac{C}{M^2}.
\end{align*}
Since the r.h.s. above can be made arbitrarily small by choosing $M$ large, the proof of the lemma is complete.

\end{proof}

\begin{proof}[Proof of Theorem \ref{est}]
 It suffices to check the two  conditions \eqref{cnc1} and \eqref{cnc2} of Lemma \ref{general} with $\R_{n,\theta}=\Q_{n,f,\theta}$ and $G_n(\pi,\theta)=PL_n(\pi,\theta)$. For checking (\ref{cnc1})  an exchangeable pair is constructed. 
 \\

  Consider the following exchangeable pair of permutations $(\pi,\pi')$ on $S_n$ constructed as follows:
  
Pick $\pi$ from $\Q_{n,f,\theta}$. To construct $\pi'$,  first pick a pair $(I,J)$ uniformly from 
the set of all ${n\choose 2}$ pairs
$\{(i,j):1\leq i<j\leq n\}$, and replace $(\pi(I),\pi(J))$ by an independent pick from the conditional distribution $(\pi(I),\pi(J)|\pi(k),k\neq I,J)$. 
By a simple calculation, the probabilities turn out to be 
\begin{align*}
(\pi'(I),\pi'(J))=(\pi(I),\pi(J)) \text{ w.p. } &\Q_{n,f,\theta}(\pi(I')=\pi(I),\pi(J')=\pi(J)|\pi(k),k\neq I,J)\\
=&\frac{1}{1+e^{\theta y_\pi(I,J)}},\\
=(\pi(J),\pi(I)) \text{ w.p. }& \Q_{n,f,\theta}(\pi(I')=\pi(J),\pi(J')=\pi(I)|\pi(k),k\neq I,J)\\
=& \frac{e^{\theta y_\pi({I,J)}}}{1+e^{\theta y_\pi(I,J)}}.
\end{align*}
Set $\pi'(i)=\pi(i)$ for all $i\neq I,J$. It can be readily checked that $(\pi,\pi')$ is indeed an exchangeable pair. Also  defining \[W(\pi):=\sum_{i=1}^nf(i/n,\pi(i)/n),\text{ and } F(\pi,\pi'):=W(\pi)-W(\pi')\] one can check from the construction of $(\pi,\pi')$  that $$\mathbb{E}_{\Q_{n,f,\theta}}[F(\pi,\pi')|\pi]=W(\pi)-\E _{\Q_{n,f,\theta}}[W(\pi')|\pi]=\frac{1}{N_n}PL_n(\pi,\theta),$$
where $P_n(\pi,\theta)$ is as defined in the statement of the Lemma, and $N_n:=\frac{n(n-1)}{2}$.
 Thus  \begin{align*}
\mathbb{E}_{\Q_{n,f,\theta}}PL_n(\pi,\theta)^2=&N_n\E_{\Q_{n,f,\theta}}PL_n(\pi,\theta)[\E_{\Q_{n,f,\theta}}F(\pi,\pi')|\pi]\\=&N_n\mathbb{E}_{\Q_{n,f,\theta}}PL_n(\pi,\theta)F(\pi,\pi')\\
=&N_n\mathbb{E}_{\Q_{n,f,\theta}}PL_n(\pi',\theta)F(\pi',\pi)\\
=&-N_n\mathbb{E}_{\Q_{n,f,\theta}}PL_n(\pi',\theta)F(\pi,\pi')\\
=&\frac{N_n}{2}\E_{\Q_{n,f,\theta}}(PL_n(\pi,\theta)-PL_n(\pi',\theta))F(\pi,\pi')
\end{align*}
where the third line uses the exchangeability of  $(\pi,\pi')$, and the fourth line uses antisymmetry $F$, and the last line is obtained by adding the second and fourth lines together and dividing by 2. This readily implies 
 \begin{align}\label{p1}\mathbb{E}_{\Q_{n,f,\theta}}PL_n(\pi,\theta)^2=&\E_{\Q_{n,\theta}}V_n(\pi)\end{align} 

where $V_n(\pi)=\frac{N_n}{2} \mathbb{E}_{\Q_{n,f,\theta}}[(PL_n(\pi,\theta)-PL_n(\pi',\theta))F(\pi,\pi')|\pi)$.
Letting $\pi^{ij}$ denote  $\pi$ with the elements $(\pi(i),\pi(j))$ swapped,  $V_n(\pi)$ can be written as
\begin{align}\label{eq:v_pi} 
V_n(\pi)=\frac{1}{2}\sum_{1\le i<j\le n}&\Big[PL_n(\pi,\theta)-PL_n(\pi^{ij},\theta)\Big]\frac{y_\pi(i,j)e^{\theta y_\pi(i,j)}}{{1+e^{\theta y_\pi(i,j)}}}.
\end{align}
Also setting $M:=4\sup_{[0,1]^2}|f|$  for any $(i,j)$ one has
\begin{align*}
|PL_n(\pi,\theta)-PL_n(\pi^{ij},\theta)|\leq 4nM,  
\end{align*}
using the fact that $|y_\pi(i,j)|\le M$.
This along with equation \eqref{eq:v_pi} gives $|V_n(\pi)|\leq 4n^3M^2$, which, along with (\ref{p1}), completes the proof of (\ref{cnc1}) with $C=4M^2$.

Proceeding to check (\ref{cnc2}) one has
\begin{align*}
-\frac{1}{n^2}PL_n'(\pi,\theta)=\frac{1}{n^2}\sum_{1\leq i<j\leq n}y_\pi(i,j)^2\frac{e^{\theta y_\pi(i,j)}}{1+e^{\theta y_\pi(i,j)}}\frac{1}{1+e^{\theta y_\pi(i,j)}}\geq \frac{e^{-|\theta| M}}{8n^2}\sum_{i,j=1}^ny_\pi(i,j)^2,
\end{align*}
where  the last inequality  again uses  $|y_\pi(i,j)|\leq M$.
Since the function  $g:[0,1]^4\mapsto \R$  defined by $$g((x_1,y_1),(x_2,y_2)):=\Big[f(x_1,y_1)+f(x_2,y_2)-f(x_1,y_2)-f(x_2,y_1)\Big]^2$$  is continuous, it follows that $\nu_\pi\times \nu_\pi\stackrel{w}{\rightarrow}\mu_{f,\theta_0}\times \mu_{f,\theta_0}$ in probability by part (b) of theorem \ref{l1}. This gives
\begin{align*}
&\frac{1}{n^2}\sum_{i=1}^ny_\pi(i,j)^2
=\frac{1}{n^2}\sum_{i,j=1}^ng((i/n,\pi(i)/n),(j/n,\pi(j)/n))
=(\nu_\pi\times \nu_\pi)(g)\\
&\stackrel{p}{\rightarrow}\int_{[0,1]^4}\Big[f(x_1,y_1)+f(x_2,y_2)-f(x_1,y_2)-f(x_2,y_1)\Big]^2d\mu_{f,\theta_0}(x_1,y_1)d\mu_{f,\theta_0}(x_2,y_2)=:\alpha(\theta),\text{ say}.
\end{align*}
If $\alpha(\theta)=0$, then $f(x_1,y_1)+f(x_2,y_2)=f(x_1,y_2)+f(x_2,y_1)$ almost surely. On integrating with respect to $x_2,y_2$ and using the fact that $f\in \cC$ gives $f(x_1,y_1)\equiv 0$, a  contradiction. Thus $\alpha(\theta)>0$, 
 and so  (\ref{cnc2}) holds with $\lambda(\theta)=e^{-M|\theta|}\alpha(\theta)/16$. Thus both conditions of Lemma \ref{general} hold,  and so  the conclusion follows.

\end{proof}

\begin{proof}[Proof of Proposition \ref{mallow}]

\begin{enumerate}[(a)]
\item
First it will be shown that  $\mu\mapsto \theta [\mu\times \mu](h)/2$ is continuous with respect to weak topology on $\sM$. Since $\sM$ is separable, it suffices to work with sequences, and it suffices to check the following:
\[\mu_k\in \sM, \mu_k\stackrel{w}{\rightarrow}\mu\Rightarrow (\mu_k\times \mu_k)(x_1\leq x_2,y_1\leq y_2)\rightarrow \mu(x_1\leq x_2,y_1\leq y_2)\]
But this follows from the fact that the boundary  of the set $\{x_1\leq x_2,y_1\leq y_2\}$ is 
a subset of $\{x_1= x_2,0\le y\le 1\}\cup \{0\le x\le 1, y_1=y_2\}$, and $\mathbb{P}(X_1= X_2)=0$ where $X_1,X_2$ are i.i.d. with distribution $U[0,1]$.
Thus $\mu\mapsto \theta[\mu\times \mu](h)/2$ is continuous on $\sM\supset \{\mu:\overline{I}(\mu)<\infty\}$. 
\\

Now, a similar computation as in the proof of Theorem \ref{l1} gives
\[e^{C_n(\theta)-C_n(0)}=\frac{1}{n!}\sum_{\pi\in S_n}e^{\frac{\theta}{n}\sum_{1\leq i< j\leq n}h((i/n,\pi(i)/n),(j/n,\pi(j)/n))}=\mathbb{E}_{\mathbb{P}_n}e^{n\frac{\theta}{2} [\nu_\pi\times \nu_\pi](h)}.\]

It  then follows by an application of Varadhan's Lemma (\cite[Theorem 4.3.1]{DZ}) along with theorem \ref{ldp}  (on noting that the proof of Varadhan's lemma goes through as long as the function $\mu\mapsto \theta(\mu\times \mu)(h)/2$ is continuous on the set $\{\overline{I}(\mu)<\infty\}$), that
$$C(\theta)=\lim_{n\rightarrow\infty}\frac{C_n(\theta)-C_n(0)}{n}=\sup_{\mu\in \cM}\Big\{\frac{\theta}{2}(\mu\times \mu)(h)-D(\mu||u)\Big\}.$$

The optimization problem was solved in \cite{SS} to show that there is a unique maximizer in $\cM$, and it has the density $u_\theta(.,.)$ with respect to Lebesgue measure. Plugging in the formula for $u_\theta(.,.)$ gives the formula for $C(\theta)$.

%

\item
Since in this case the function $C(\theta)$ is convex, differentiable with a derivative which is continuous and monotone increasing, consistency of $\widetilde{\theta}_{LD}$ and $\widetilde{\theta}_{ML}$ follow from similar arguments as in Corollary \ref{thm:ldmle}.


%
%
%
%
%
%
%
%
%
\end{enumerate}

\end{proof}

\section{Appendix: Proof of Theorem \ref{ldp}}\label{appen:2}
The proof is carried out  by using \cite[Theorem 4.1.11]{DZ} by choosing a suitable base for the weak topology.

\begin{defn}\label{def0}
Let  $\cM_{k,n}$ denote the number of non negative integer valued $k\times k$ matrices with  $r^{th}$ row sum  equal to $M_r:=\lceil \frac{nr}{k}\rceil-\lceil\frac{n(r-1)}{k}\rceil$ and $s^{th}$ column sum equal to $\lceil \frac{ns}{k}\rceil-\lceil\frac{n(s-1)}{k}\rceil$, i.e. 
\[\cM_{k,n}:=\Big(M\in\mathbb{N}_0^{k^2}:\sum_{s=1}^kM_{rs}=M_r, \sum_{r=1}^kM_{rs}=M_s\Big),\]
where $\mathbb{N}_0:=\mathbb{N}\cup \{0\}$.
 Note that any $M\in \cM_{k,n}$ satisfies $\sum_{r,s=1}^kM_{rs}=n$.  Recall the matrix $M(\pi)$ defined in  \eqref{eq:define_M} as the $k\times k$ matrix with $M_{rs}(\pi)=n\nu_\pi[T_{rs}]$.
\end{defn}
If $\pi$ is random, $M(\pi)$ is a random matrix. The first lemma gives the  distribution of  $M(\pi)$  when $\pi\sim \P_n$.

\begin{lem}\label{aux}
The  distribution of $M(\pi)$ is given by 
\[ \mathbb{P}_n(M(\pi)=M)=\frac{\Big(\prod_{r=1}^kM_{r}!\Big)^2}{n!\prod_{r,s=1}^kM_{rs}!}\]
if $M\in \cM_{k,n}$, 
and $0$ otherwise. 

\end{lem}

\begin{proof}

Since \[M_{r,s}(\pi)=\sum_{i=1}^n1\Big\{\Big\lceil \frac{ki}{n}\Big\rceil =r,\Big\lceil \frac{k\pi(i)}{n}\Big\rceil =s\Big\},\]
it follows that \[\sum_{s=1}^kM_{r,s}(\pi)=\sum_{i=1}^n1\Big\{\Big\lceil \frac{ki}{n}\Big\rceil =r\Big\}=M_r,\]
and so any valid configuration $M$ is in $\cM_{k,n}$.  So fixing a particular configuration $M\in  \cM_{k,n}$, the number of possible permutations $\pi$ compatible with this configuration can be computed as follows:
\\

For the $r^{th}$ row there are $M_r$ choices of indices $i$, and that can be allocated in boxes $\{T_{r,s}\}_{s=1}^k$ in $M_r!/\prod_{s=1}^kM_{rs}!$ ways, so that box $T_{r,s}$ receives $M_{r,s}$  indices. Taking a product over $r$, the number of ways to distribute the indices over the boxes is \[\frac{\prod_{r=1}^kM_r!}{\prod_{r,s=1}^kM_{rs}!}\]
Similarly, the number of ways to distribute the targets $\{\pi(i)\}$ such that box $T_{r,s}$ receives $M_{rs}$ targets is  \[\frac{\prod_{s=1}^kM_s!}{\prod_{r,s=1}^kM_{rs}!}\]
Finally after the above distribution box $T_{r,s}$ has $M_{rs}$ indices and $M_{rs}$ targets, which can then be permuted freely, and so the total number of permutations compatible with  any such distribution of indices and targets is \[\prod_{r,s=1}^kM_{rs}!\] Combining, the total number of possible permutations $\pi$ satisfying $M(\pi)=M$  is given by \[\frac{\prod_{r=1}^kM_r!\prod_{s=1}^kM_s!}{\prod_{r,s=1}^kM_{rs}!}\] Since the total number of permutations in $n!$, the proof of the claim is complete. 
\end{proof}

\begin{remark}
Note that in the above proposition the row and column sums of the matrix $M$ are free of $\pi$. The distribution of $M$ is a multivariate generalization of the hypergeometric distribution, commonly known as the Fisher-Yates distribution. This distribution  arises in statistics while testing for independence in a 2-way table in the works of Diaconis-Efron (\cite{pd_efron_II},\cite{pd_efron_I}).
\end{remark}
Before proceeding the following definitions are needed. The first definition gives  a base for the weak topology on $\sM$. 
\begin{defn}\label{def1}
For any $\mu\in \sM$ define $P_{k,\mu}\in [0,1]^{k^2}$  by setting $P_{k,\mu}(r,s):=\mu(T_{r,s})$. Note that $T_{rs}$ is a $\mu $ continuity set, and so the map $\mu\mapsto P_{k,\mu}$ is continuous on $\sM$ with respect to weak convergence. 


One can now define a base for the weak topology on $\sM$ as follows:
Fix  $k\in \mathbb{N},\epsilon>0,\mu_0\in \sM$, and  define the  set $$\ub(\epsilon):=\{\mu\in\sM:||P_{k,\mu}-P_{k,\mu_0}||_\infty<\epsilon\},$$ 
where  $$||P_{k,\mu}-P_{k,\mu_0}||_\infty:=\max_{1\le r,s\le k}|P_{k,\mu}(r,s)-P_{k,\mu_0}(r,s)|.$$
Since $\mu\mapsto P_{k,\mu}$ is continuous, the set $\ub(\epsilon)$ is open in $\sM$.  Recall the definition of $\cM_k$ from definition \ref{def:mk}  and that for any $\mu\in \cM $ one has $\P_{k,\mu}\in \cM_k$. Thus the operation $A\mapsto p_A$ introduced in defintiion \ref{def:pA} maps a matrix to a probability measure, and the operation  $\mu\mapsto P_{k,\mu}$ maps a probability measure to a matrix.
\end{defn}
\begin{ppn}
The collection
$$\sM_0:=\{\ub(\epsilon):k\in\mathbb{N};\epsilon>0,\mu_0\in \sM\}$$ is a base for the weak convergence on $\sM$. 
\end{ppn}

\begin{proof}
One needs to verify that given any $\mu_0$ and an open set $U$ containing $\mu_0$, there is an element $U_0$ from this collection $\sM_0$ such that $\mu_0\in U_0\subset U$. If not, then in particular the set $\sM[k,\mu_0](1/k)$ is not contained in $U$ for any $k$, and so there exists $\mu_k\in \sM[k,\mu_0](1/k)\cap U^c$.  Then for any function $f$ which is continuous on the unit square, one has
\begin{align*}
|\mu[f]-\mu_k[f]|\le \max_{[0,1]^2}|f|||P_{k,\mu}-P_{k,\mu_0}||_\infty+2\sup_{|x_1-x_2|,|y_1-y_2|\le 1/k}|f(x_1,y_1)-f(x_2,y_2)|,
\end{align*}
which goes to $0$ as $k$ goes to $\infty$. Thus $\mu_k$ converges weakly to $\mu$, and since $U$ is open, one has that $\mu_k\in U$ for all large $k$. This is a contradiction to the assumption that $\mu_k\notin U$, and so completes the proof.
\end{proof}
 
 This reduces the analysis of measures to the analysis of $k\times k$ matrices for a large but fixed $k$.
 \begin{defn}
  For $\mu_0\in \cM$ define a set $\lbl(\epsilon)\subset \mathcal{M}_k$ as 
\begin{align*}
\lbl(\epsilon):=\{A\in \cM_k:||M-P_{k,\mu_0}||_\infty<\epsilon\}.
\end{align*}
Since $M(\pi)\in \cM_{k,n}$ is an integer valued matrix, all configurations in $ \lbl(\epsilon)$ cannot be attained by setting $A=M(\pi)/n$. Define $\llb(\epsilon)$ to be the set of all $M\in \cM_{k,n}$ such that $\frac{1}{n}M\in \lbl$. More precisely, $\llb(\epsilon)$ is defined by
\begin{align*}
\llb(\epsilon):=\cM_{k,n}\cap n\lbl(\epsilon)=\Big\{M\in\cM_{k,n}:||\frac{1}{n}M-P_{k,\mu_0}||_\infty<\epsilon\}.
\end{align*}
\end{defn}
The following lemma gives an estimate of the probability that $M(\pi)\in \llb(\epsilon)$. 
\begin{lem}\label{dist}

\[\lim_{n\rightarrow\infty}\frac{1}{n}\log \mathbb{P}_n(M(\pi)\in \llb(\epsilon))=-\inf_{A\in\lbl(\epsilon)}D(p_A||p_{U_k}),\]
where $p_A$ is as in definition \ref{def:pA}.
\end{lem}

\begin{proof}
For the proof, first assume  that 
\begin{align}\label{claim}
\lim_{n\rightarrow\infty}\min_{M\in \llb(\epsilon)}D(p_{M/n}||p_{U_k})=\inf_{A\in\lbl(\epsilon)}D(p_A||p_{U_k}),
\end{align}
where the definition of $p_A$ to matrices $A$ whose row/column sums need not equal $1/k$, to accomodate for the fact that for any $M\in \llb(\epsilon)$ the matrix $\frac{1}{n}M$ will not satisfy this exactly.
The proof of (\ref{claim}) is deferred till the end of the lemma.

For the lower bound, note that
\begin{align*}
\mathbb{P}_n(M(\pi)\in\llb(\epsilon))\geq &\max_{M\in  \llb(\epsilon)}\mathbb{P}_n(M(\pi)=M)\\
= &\max_{ M\in \llb(\epsilon)}\frac{\Big(\prod_{r=1}^kM_r\Big)^2}{n!\prod_{r,s=1}^kM_{rs}!}
\end{align*}
where  the second step uses Lemma \ref{aux}. Now, Stirling's formula gives that there exists $C<\infty$ such that 
\begin{align*}
  |\log n!-n\log n+n|=&0\text{ if }n=0\\
  =&1\text{ if }n=1\\
  \le &C\log n\text{ if }n\ge 2,
  \end{align*}
 and so  \begin{align*}
\frac{1}{n}\log\mathbb{P}_n(M(\pi)\in \llb(\epsilon))\geq -\min_{M\in \llb(\epsilon)}D(p_{M/n}||p_{U_k})-\frac{C_k\log n}{n}
\end{align*}
for some constant $C_k<\infty$. 
On taking limits using (\ref{claim})  completes the proof of the lower bound.

For the upper bound  note that
\begin{align*}
\mathbb{P}_n(M(\pi)\in \llb(\epsilon))\leq  &\Big({{n+k^2-1}\atop{k^2-1}}\Big)\max_{M\in \llb(\epsilon)}\mathbb{P}_n(M(\pi)=M)\\
\leq& (n+k^2)^{k^2}\max_{ M\in\llb(\epsilon)}\mathbb{P}_n(M(\pi)=M),\end{align*} since any valid configuration $M$ is a non negative integral solution of the equation $\sum_{r,s=1}^kM_{rs}=n$. Thus proceeding as before it follows that
\begin{align*}
\frac{1}{n}\log\mathbb{P}_n(M(\pi)\in \llb(\epsilon))\leq -\min_{M\in \llb(\epsilon)}D(p_{M/n}||p_{U_k})+\frac{C_k'\log n}{n}
\end{align*}
for some other $C_k'<\infty$, which on taking limits using (\ref{claim}) completes the proof of the upper bound.
 \\
 
 It thus remains to prove (\ref{claim}). To this effect, let $M^{(n)}$ denote the minimizing configuration on the l.h.s. of (\ref{claim}). Then $\frac{1}{n}M^{(n)}$ is a sequence in the compact set $\{A:A_{rs}\geq 0:\sum_{r,s=1}^kA_{rs}=1\}$, and any convergent subsequence converges to a point in $ \overline{\lbl(\epsilon)}$. Thus 
 \[\liminf_{n\rightarrow\infty}\min_{M\in\llb(\epsilon)}D(p_{M/n}||p_{U_k})\geq \inf_{A\in \overline{\lbl(\epsilon)}}D(p_A||p_{U_k})=\inf_{A\in \lbl(\epsilon)}D(p_A||p_{U_k}),\]
where the last equality follows from since $A\mapsto D(p_A||p_{U_k})$ is continuous, completing the proof of the lower bound in (\ref{claim}).

Proceeding to prove the upper bound, it suffices to prove that  for any $A\in\lbl(\epsilon)$ there exists a sequence $M^{(n)}\in\llb(\epsilon)$ such that $\frac{1}{n}M^{(n)}$ converges to $A$ as $n\rightarrow\infty$. To this effect, let  $\mu\in\sM$ be such that $P_{k,\mu}=A$. (It is easy to check that such a $\mu$ always exists for any $A\in \cM_k$). By \cite[Lemma 4.2]{HKMRS} and \cite[Lemma 5.3]{HKMRS}  there exists a sequence of permutations $\{\sigma_n\}_{n\ge 1}$ with $\sigma_n\in S_n$ such that $\nu_{\sigma_n}$ converges weakly to $\mu$, and so  setting $M^{(n)}=M(\sigma_n)$ one has that $M^{(n)}\in \cM_{k,n}$ and $\frac{1}{n}M^{(n)}\rightarrow P_{k,\mu}=A$. Also the set
$$W_k:=\{B\in[0,1]^{k^2}:||B-P_{k,\mu_0}||_\infty<\epsilon\} $$ is open, and since $A\in W_k$, it follows that $\frac{1}{n}M^{(n)}\in W_k$ for all large $n$. Since $\llb(\epsilon)=nW_k\cap \cM_{k,n}$, the proof of (\ref{claim}) is complete.
\end{proof}
The next lemma derives another technical estimate using Lemma \ref{dist}. This lemma will be used to prove Theorem \ref{ldp}.
\begin{lem}\label{limit}
For any set $\ub(\epsilon)$ 
 one has $$\lim_{n\rightarrow\infty}\frac{1}{n}\log \mathbb{P}_n(\mu_\pi\in\ub(\epsilon))=-\inf_{A\in\lbl(\epsilon)} D(p_A||p_{U_k}).$$ 

\end{lem}

\begin{proof}
First note that 
\[||P_{k,\mu_\pi}-\frac{1}{n}M(\pi)||_\infty\leq \frac{2}{n}.\] Indeed, since each square $T_{rs}$ has four boundaries each of which intersect in exactly one row/column of the $n\times n$ partition of the unit square, the two quantities above can differ only if there is an element on one of these rows/columns. Since each such square has probability $1/n$ under $\mu_\pi$, the maximum difference can be at most $2/n$.
\\
Thus 
 for any $\delta\in (0,\epsilon)$ and all $n$ large enough,
$$\mathbb{P}_n(\mu_\pi\in \ub(\epsilon)) \geq \mathbb{P}_n(M(\pi)\in \llb(\epsilon-\delta))$$
Using Lemma \ref{dist} gives 
\begin{align*}
\liminf_{n\rightarrow\infty}\frac{1}{n}\log \mathbb{P}_n(M(\pi)\in \llb(\epsilon-\delta))\geq -\inf_{A\in \lbl(\epsilon-\delta)}D(p_A||p_{U_k}).
\end{align*}
Letting $\delta\downarrow0$  gives
\begin{align}\label{claim1}
\liminf_{n\rightarrow\infty}\frac{1}{n}\log \mathbb{P}_n(\ub(\epsilon))\geq -\inf_{A\in\lbl(\epsilon))}D(p_A||p_{U_k}).
\end{align}
A similar argument gives
\begin{align*}
\limsup_{n\rightarrow\infty}\frac{1}{n}\log \mathbb{P}_n(M(\pi) \llb(\epsilon+\delta))\leq -\inf_{A\in\lbl(\epsilon+\delta))}D(p_A||p_{U_k}),
\end{align*}
from which, letting $\delta\downarrow 0$ gives 
\begin{align}\label{claim2}
\limsup_{n\rightarrow\infty}\frac{1}{n}\log \mathbb{P}_n(\ub(\epsilon))\leq -\inf_{A\in\overline{\lbl(\epsilon)})}D(p_A||p_{U_k}).
\end{align}
Combining (\ref{claim1}) and (\ref{claim2}) 
gives $$\lim_{n\rightarrow\infty}\frac{1}{n}\log\mathbb{P}_n(\ub(\epsilon))=-\inf_{A\in\lbl(\epsilon))}D(p_A||p_{U_k}),$$ using the continuity of  $A\mapsto D(p_A||p_{U_k})$. This completes the proof of the lemma. 
\end{proof}
\begin{proof}[Proof of Theorem \ref{ldp}]
Since $\sM_0$ is a base for the weak topology on $\sM$, by Lemma \ref{limit} and \cite[Theorem 4.1.11]{DZ} it follows that $\mathbb{P}_n$ follows a weak ldp with the rate function 
$$I(\mu)=\sup_{\ub(\epsilon)\ni\mu}\quad \inf_{A\in \lbl(\epsilon)}D(p_A||p_{U_k}).$$ 
 Also since $\sM$ is compact it follows that  full ldp holds with the good rate function $I(.)$.
It thus remains to prove that $I(\mu)=D(\mu||u)$. 
 To this effect, first note that $\mu\in \sM[k,\mu](1/k)$, and so 
 \[I(\mu)\ge \liminf_{k\rightarrow\infty}\inf_{A\in \overline{\cV[k,\mu](1/k)}}D(p_A||p_{U_k})=\liminf_{k\rightarrow\infty}D(p_{A_k}||p_{U_k}),\]
where $A_k$ denotes any minimizer of 
$A\mapsto D(p_A||p_{U_k})$ over $\overline{\cV[k,\mu](1/k)}$. But then $p_{A_k}$ converges weakly to $\mu$ as $k\rightarrow\infty$. 
 The  lower semi continuity of $D(.||.)$ then implies $I(\mu)\geq D(\mu||u)$, proving the lower bound. 
\\

 For the upper bound note that  the first supremum is over all $\sM[k,\mu_0](\epsilon)$ containing $\mu$ , and so  with $A=P_{k,\mu}\in \cV[k,\mu](\epsilon)$ one has $$I(\mu)\leq \sup_{k\ge 1}D(p_{_{P_{k,\mu}}}||p_{U_k})$$ Also note that 
 \begin{align*}
 D(\mu||u)=&\sup_{f\in B[0,1]^2}\{\int_{[0,1]^2}fd\mu-\log \int_{[0,1]^2}e^f du\},\\
 D(p_{_{P_{k,\mu}}}||p_{U_k})=&\sup_{f\in B_k[0,1]^2}\{\int_{[0,1]^2}fd\mu-\log \int_{[0,1]^2}e^f du\},\
\end{align*}
 where $B[0,1]^2$ denotes the set of all bounded measurable functions on $[0,1]^2$, and $B_k[0,1]^2$ denotes the subset of $B[0,1]^2$ which is constant on every $T_{rs},1\leq r,s\leq k$. Indeed, both the  results follows from \cite[Lemma 6.2.13]{DZ}. Consequently  $\sup_{k\ge 1} D(p_{_{P_{k,\mu}}}||p_{U_k})\leq D(\mu||u)$, thus completing the proof of the upper bound.

   \end{proof}

   \section{Acknowledgement}

   This paper benefitted from helpful discussions with Persi Diaconis, Amir Dembo, Sourav Chatterjee, Susan Holmes, Bhaswar Bhattacharya, and Austen Head.  I would like to thank Maxwell Grazier G'Sell for help with acquiring the draft lottery data. 
   
   The contents of this paper also appear in the author's Phd thesis advised by Persi Diaconis.

\end{document}